\documentclass[12pt]{amsart}
\usepackage[utf8]{inputenc}

\usepackage{amsmath, amssymb, amsthm, mathrsfs, graphicx,float, mathtools}
\usepackage{comment}
\usepackage[shortlabels]{enumitem}
\usepackage[numbers]{natbib}
\usepackage[margin=1in]{geometry}
\usepackage{stmaryrd} 
\usepackage{xcolor}
\usepackage{tikz-cd}
\usepackage{tikz}

\usepackage{yfonts}

\graphicspath{{graphics/}}

\newtheorem{theorem}{Theorem}
\newtheorem{proposition}[theorem]{Proposition}
\newtheorem{lemma}[theorem]{Lemma}
\newtheorem{corollary}[theorem]{Corollary}

\theoremstyle{definition}
\newtheorem{definition}[theorem]{Definition}
\newtheorem{example}[theorem]{Example}
\newtheorem{remark}[theorem]{Remark}

\numberwithin{theorem}{section}

\makeatletter
\providecommand{\leftsquigarrow}{%
  \mathrel{\mathpalette\reflect@squig\relax}%
}
\newcommand{\reflect@squig}[2]{%
  \reflectbox{$\m@th#1\rightsquigarrow$}%
}
\makeatother

\renewcommand{\phi}{\varphi} 
\renewcommand{\vec}{\mathbf} 

\newcommand{\id}{\mathrm{id}}

\newcommand{\ol}[1]{\overline{#1}}

\newcommand{\emphbf}[1]{\textbf{#1}}

\definecolor{sebgreen1}{rgb}{0.019,0.317,0.149}
\definecolor{sebgreen2}{rgb}{0.784,0.952,0.780}

\definecolor{chrisyellow1}{rgb}{1,0.73,0}
\definecolor{chrisyellow2}{rgb}{0.99,0.99,0.75}

\definecolor{sebdarkgreen}{rgb}{0.019,0.317,0.149}
\definecolor{sebgreen}{RGB}{36,157,55}
\definecolor{seblightgreen}{rgb}{0.784,0.952,0.780}
\definecolor{sebblue}{RGB}{0,123,194}
\definecolor{seblightblue}{RGB}{194,242,255}

\usepackage[pagebackref]{hyperref}
\hypersetup{
  colorlinks   = true,          
  urlcolor     = sebblue,          
  linkcolor    = purple,          
  citecolor   = sebblue             
}

\newcommand{\mm}{\mathfrak{m}}

\renewcommand{\AA}{\mathbb{A}}

\newcommand{\GG}{\mathbb{G}}
\newcommand{\NN}{\mathbb{N}}
\newcommand{\PP}{\mathbb{P}}

\newcommand{\QQ}{\mathbb{Q}}
\newcommand{\ZZ}{\mathbb{Z}}

\newcommand{\OO}{\mathscr{O}}

\newcommand{\UU}{\mathcal{U}}

\newcommand{\Ann}{\mathrm{Ann}}

\newcommand{\Aut}{\mathrm{Aut}}

\newcommand{\Cond}{\mathrm{Cond}}
\newcommand{\GL}{\mathrm{GL}}

\newcommand{\Grass}{\mathrm{Gr}}

\newcommand{\IH}{\mathrm{IH}}

\newcommand{\moduli}{\mathcal{M}}

\newcommand{\Proj}{\mathrm{Proj}\,}
\newcommand{\rdim}{\mathrm{r.dim\;}}

\newcommand{\Spec}{\mathrm{Spec}\,}
\newcommand{\Spine}{\mathrm{Spine}}

\renewcommand{\tilde}{\widetilde}

\newcommand{\BTer}{\mathrm{BTer}}
\newcommand{\Ter}{\mathrm{Ter}}

\newcommand{\SSpec}[2]{\underline{\mathrm{Spec}}_{#1}\, #2}

\newcommand{\Grp}{\mathrm{Grp}}

\newcommand{\Sch}{\mathrm{Sch}}
\newcommand{\Set}{\mathrm{Set}}

\theoremstyle{definition}

\newcommand{\smoothabilityplotcorrected}{
\begin{tikzpicture}[scale=0.35]
    \draw[->] (0,0) -- (33,0) node[right] {$m$};
    \draw[->] (0,0) -- (0,29) node[above] {$g$};
    \foreach \m in {1,...,32} {\node[below] at (\m,-0.1) {\tiny \m};}
    \foreach \g in {0,...,28} {\node[left] at (0,\g) {\tiny \g};}

    \definecolor{darkblue}{RGB}{0,0,220}
    \definecolor{lightblue}{RGB}{173,216,230}
    \definecolor{forestgreen}{RGB}{60,167,18}
    \definecolor{monopurple}{RGB}{148,0,211} 

    \foreach \g in {0,...,28} {
        \foreach \m in {1,...,32} {
            \pgfmathsetmacro{\colorname}{0}

            \ifnum \g<3 \pgfmathsetmacro{\colorname}{1} \fi
            \ifnum \g=3 \ifnum \m<15 \pgfmathsetmacro{\colorname}{1} \fi \fi

            \ifnum \colorname=0 \ifnum \m=1
                \ifnum \g>21
                    \pgfmathsetmacro{\colorname}{2}   
                \else
                    \ifnum \g>15
                        \pgfmathsetmacro{\colorname}{7} 
                    \fi
                \fi
            \fi \fi

            \ifnum \colorname=0
                \ifnum \g=3 \ifnum \m>14 \pgfmathsetmacro{\colorname}{3} \fi \fi
                \ifnum \g=4 \ifnum \m>12 \pgfmathsetmacro{\colorname}{3} \fi \fi
                \ifnum \g=5 \ifnum \m>12 \pgfmathsetmacro{\colorname}{3} \fi \fi
                \ifnum \g>5
                    \pgfmathsetmacro{\mg}{2*\g+1}
                    \pgfmathparse{\m >= int(\mg) ? 1 : 0}
                    \ifnum \pgfmathresult=1 \pgfmathsetmacro{\colorname}{3} \fi
                \fi
            \fi

            \ifnum \colorname=0 \ifnum \g>2
                \pgfmathsetmacro{\threshold}{\g + 5 + 7/(\g - 2)}
                \pgfmathsetmacro{\upper}{2*\g+1}
                \pgfmathparse{(\m >= \threshold) && (\m < \upper) ? 1 : 0}
                \ifnum \pgfmathresult=1 \pgfmathsetmacro{\colorname}{4} \fi
            \fi \fi

            \ifnum \colorname=0
                \ifnum \g=23 \ifnum \m=29 \pgfmathsetmacro{\colorname}{4} \fi \fi
                \ifnum \g=25 \ifnum \m=30 \pgfmathsetmacro{\colorname}{4} \fi \fi
            \fi

            \ifnum \colorname=0
                \pgfmathtruncatemacro{\nval}{\m-\g}                 
                \ifnum \nval>0
                    \pgfmathtruncatemacro{\cval}{\nval*(\nval+1)/2} 
                    \ifnum \m<\numexpr\cval+1\relax                 
                        \ifodd\g \def\evar{1}\else \def\evar{0}\fi  
                        \ifnum \g=8 \def\evar{1}\fi                 
                        \pgfmathtruncatemacro{\lval}{(\g-2)*(\m-\g-5)}
                        \ifnum \lval>\numexpr6*\evar\relax          
                            \pgfmathsetmacro{\colorname}{6}
                        \fi
                    \fi
                \fi
            \fi

            \ifnum \colorname=0
                \pgfmathparse{3*\g < \m ? 1 : 0}
                \ifnum \pgfmathresult=1
                    \pgfmathparse{(\m - \g)*\g >= 3*\g - 3 + 2*\m ? 1 : 0}
                    \ifnum \pgfmathresult=1 \pgfmathsetmacro{\colorname}{5} \fi
                \else
                    \pgfmathsetmacro{\b}{Mod(\g + \m, 4) == 0 ? 0 : (Mod(\g + \m, 4) == 2 ? 4 : 1)}
                    \pgfmathparse{(\g+\m)^2 - \b >= 24*\g - 24 + 16*\m ? 1 : 0}
                    \ifnum \pgfmathresult=1 \pgfmathsetmacro{\colorname}{5} \fi
                \fi
            \fi

            \ifnum \colorname>0
                \ifnum \colorname=1 \def\dotcolor{white}      \fi
                \ifnum \colorname=2 \def\dotcolor{black}      \fi
                \ifnum \colorname=3 \def\dotcolor{darkblue}   \fi
                \ifnum \colorname=4 \def\dotcolor{lightblue}  \fi
                \ifnum \colorname=5 \def\dotcolor{orange}     \fi
                \ifnum \colorname=6 \def\dotcolor{forestgreen}\fi
                \ifnum \colorname=7 \def\dotcolor{monopurple} \fi
                \filldraw[fill=\dotcolor,draw=black] (\m,\g) circle (0.3);
            \fi
        }
    }

\end{tikzpicture}
}

\title{Close connectedness of the moduli stack of reduced curves}
\author{Sebastian Bozlee}
\date{\today}

\begin{document}

\begin{abstract}
We prove that the moduli stack of all reduced $n$-pointed curves is ``closely connected'' in characteristic zero, in the sense that each irreducible component of the stack intersects the component of smoothable curves. We achieve this by performing a detailed study of Ishii's territories, moduli schemes parametrizing reduced curve singularities together with a normalization map. We give explicit equations for territories, bound their dimensions, describe certain functoriality properties, and study the action of several groups on territories. Along the way, we prove the existence of non-smoothable reduced curve singularities in new ranges, generalizing work of Mumford, Pinkham, Greuel, and Stevens.
\end{abstract}

\maketitle

\tableofcontents

\section{Introduction}

Using Ishii's theory of territories, it has been shown that the moduli stack $\UU_{g,n}$ of all reduced $n$-pointed curves of genus $g$ is connected \cite{bozlee_connectedness}. This paper improves the result to the following in characteristic 0, which we call ``close connectedness.''

\begin{theorem} \label{thm:close_connectedness_intro} (= Theorem \ref{thm:close_connectedness})
Let $k$ be a field of characteristic 0. Then each irreducible component of $\UU_{g,n} \times \Spec k$ intersects the irreducible substack of smoothable curves.
\end{theorem}

Our primary tool is that of territories of curve singularities, moduli spaces parametrizing reduced algebraic curve singularities together with a normalization map, previously studied in \cite{ishii_moduli_subrings, ishii_moduli_equisingular_curves, chris_thesis, bozlee_guevara_smyth, bozlee_connectedness}. This paper is organized around a detailed study of these moduli spaces, with an eye towards implications for the geometry of $\UU_{g,n}$. Sections \ref{sec:basic_theory}, \ref{sec:torus_action}, and \ref{sec:aut_action} contribute directly to Theorem \ref{thm:close_connectedness_intro}. We now give a definition of territories to get us started. (A more precise definition will be given in Section \ref{sec:basic_theory}.)

\begin{definition}
The \emphbf{territory of $m$-branch curve singularities of genus $g$ with conductances at most $\vec{c} = (c_1, \ldots, c_m)$}, denoted $\Ter^g_{\vec{c}}$, is the closed subscheme of the grassmannian $\Grass(n - g, A^+_{\vec{c}})$ parametrizing corank $g$ subalgebras of
\[
  A^+_{\vec{c}} \cong \ZZ[t_1, \ldots, t_m] / (t_i^{c_i},t_it_j : i, j \in \{ 1, \ldots, m \}, i \neq j),
\]
where $n$ is the rank of $A^+_{\vec{c}}$.
\end{definition}

We explain in Section \ref{sec:basic_theory} the relationship of this definition with curve singularities and give explicit equations for $\Ter^g_{\vec{c}}$. One finds that territories can have multiple irreducible components, may be generically nonreduced on these components, and the irreducible components themselves may possess singularities. This leads us to look for better behaved structures inside territories.

\bigskip

In Section \ref{sec:spine}, we consider a particularly simple part of a territory which we call the \emphbf{spine}, isomorphic to a grassmannian.

\begin{theorem}
Each territory $\Ter^g_{\vec{c}}$ contains a (possibly empty) closed subscheme $\Spine(g,\vec{c})$ isomorphic to $\Grass(c - m - g, \sum_{i=1}^m \lfloor \frac{c_i}{2} \rfloor)$, where $c = \sum_i c_i$.
\end{theorem}

We obtain dimension bounds on territories as an easy consequence.

\begin{theorem} (= Theorem \ref{thm:spine_dim}) \label{thm:spine_dim_intro} 
For any nonempty territory $\Ter^g_{\vec{c}}$ and any algebraically closed field $k$,
\[
  (c - m - g)\left(g + m - \frac{c}{2} - \frac{p}{2} \right) \leq \dim(\Ter^g_{\vec{c}} \times \Spec k) \leq (c - m - g) g
\]
where $c = \sum_i c_i$ and $p$ is the number of odd $c_i$'s.
\end{theorem}

Using these dimension bounds, we find families of singularities of dimension too large to be contained in the boundary of $\moduli_{g,n}$ in $\UU_{g,n}$, which consequently must contain non-smoothable singularities. This yields an existence result for non-smoothable singularities.

\begin{theorem} \label{thm:smoothability_intro} (= Theorem \ref{thm:smoothability})
Suppose $g \geq 0$, $m \geq 1$, and $(g,m) \neq (0,1)$.

\begin{enumerate}
  \item If $3g > m$, then there exists a reduced non-smoothable singularity of genus $g$ with $m$ branches whenever
\[
  (g + m)^2 - \beta \geq 24g - 24 + 16m
\]
where
\[
  \beta = \begin{cases}
      0 & \text{ if } g + m \equiv 0 \pmod{4} \\
      1 & \text{ if } g + m \equiv 1,3 \pmod{4} \\
      4 & \text{ if } g + m \equiv 2 \pmod{4}.
  \end{cases}
\]
  \item If $3g \leq m$, then there exists a reduced non-smoothable singularity of genus $g$ with $m$ branches whenever
  \[
    (m - g)g \geq 3g - 3 + 2m
  \]
\end{enumerate}
\end{theorem}

See Figure \ref{fig:smoothability} for a comparison to other results on the existence of non-smoothable singularities, and Remark \ref{rmk:nonsmoothability_literature} for more details. For a more complete survey of results on non-smoothable singularities, we direct the reader to \cite{stevens_nonsmoothable}.

\begin{figure}
\smoothabilityplotcorrected
\caption{The white dots indicate that all reduced curve singularities of genus $g$ and number of branches $m$ are smoothable \cite[Lemma 4-4, Proposition 4-5]{stevens_la_rabida}. The remaining dots indicate the existence of non-smoothable reduced curve singularities of type $(g,m)$. Black is due to Greuel's refinement \cite[Section 3.2]{greuel} of an argument of Mumford \cite{mumford_pathologies}. The dark blue, due to Pinkham \cite[Theorem 1.11]{pinkham_thesis}, light blue, due to Greuel \cite[Section 3.4]{greuel}, and green, due to Stevens \cite[Proposition 10]{stevens_canonical_curves} are witnessed by a singularity consisting of $m$ lines through the origin of some $\AA^n$. The purple dots are witnessed by monomial singularities, non-smoothable by a criterion of Buchweitz \cite{buchweitz}. The orange dots are due to Theorem \ref{thm:smoothability_intro}. To the author's knowledge, it remains open whether non-smoothable singularities of type $(g,m)$ exist for the blank region.
}
\label{fig:smoothability}
\end{figure}

\bigskip

In Section \ref{sec:functoriality} we articulate several morphisms between subschemes of $\Ter^g_{\vec{c}}$ corresponding to the operations of transverse union of singularities, contraction of branches of singularities, and restriction to branches. The subschemes on which these operations are well-defined yield a natural decomposition of $\Ter^g_{\vec{c}}$ into disjoint locally closed subschemes.

\begin{theorem}
Let $\vec{c} = (c_1,\ldots, c_m)$ be a vector of positive integers with sum $c$. Let $I \sqcup I'$ be a partition of $[m]$ into nonempty subsets. Let $g$ be a non-negative integer. Then for each partition $g_I + g_{I'} = g$ of $g$ into non-negative integers, there is a locally closed subscheme $\Ter^{g_I, g_{I'}}_{I}$ on which the following morphisms are defined:
\begin{enumerate}
 \item A morphism $\rho_{I'} : \Ter^{g_I, g_{I'}}_{I} \to \Ter^{g_{I'}}_{\vec{c}|_{I'}}$
induced by restricting the corresponding curve singularity to the branches labeled by $I'$.

\item A morphism $\kappa_I : \Ter^{g_I, g_{I'}}_{I} \to \Ter^{g_I}_{\vec{c}|_{I}}$
induced by contracting the branches labeled by $I'$ of the corresponding curve singularity.

 \item A morphism $\vee : \Ter^{g_I}_{\vec{c}|_I} \times \Ter^{g_{I'}}_{\vec{c}|_{I'}} \to \Ter^{g_I, g_{I'}}_I$ induced by taking the transverse union of the corresponding curve singularities.
\end{enumerate}

Moreover,
\[
  \Ter^{0,g}_I \sqcup  \Ter^{1,g-1}_I \sqcup \cdots \sqcup \Ter^{g,0}_I \to \Ter^g_{\vec{c}}
\]
is a stratification of $\Ter^g_{\vec{c}}$ into disjoint locally closed subschemes such that for each $j$, the closure of $\Ter^{j, g-j}_{I}$ in $\Ter^g_{\vec{c}}$ is contained in the union $\bigcup_{i \leq j} \Ter^{i,g-i}_I$.
\end{theorem}

We also consider the behavior of Gorenstein singularities under contraction.

\bigskip

In Section \ref{sec:isom_hilb}, we study a decomposition of territories into ``Isom-Hilb'' strata, originally studied by Guevara \cite{chris_thesis}. These strata parametrize singularities whose restriction to their branches has fixed genus.

\begin{definition} (= Definition \ref{def:isom_hilb_strata})
Let $\vec{c} = (c_1, \ldots, c_m)$ be a vector of positive integers with sum $c$.
Let $I_1 \subseteq \{1, \ldots, m\}$ be a nonempty proper subset and let $I_2$ be its complement.
Let $g$ be a non-negative integer and let $g = g_1 + g_2 + \gamma$ be a partition of $g$ into non-negative integers. The \emphbf{Isom-Hilb stratum} $\mathrm{IH}^{g_1,g_2,\gamma}_{I_1,I_2}$ is the locally closed subscheme
\[
  \mathrm{IH}^{g_1,g_2,\gamma}_{I_1,I_2} \coloneqq \Ter^{g_1 + \gamma, g_2}_{I_1} \cap \Ter^{g_2 + \gamma, g_1}_{I_2}
\]
of $\Ter^g_{\vec{c}}$.
\end{definition}

As originally observed by Guevara, these strata have an interpretation in terms of points of appropriate Hilbert and Isom schemes, enabled by Goursat's Lemma. We present a functorial version of Goursat's Lemma (Lemma \ref{lem:goursat}) and prove a scheme theoretic enhancement of Guevara's Isom-Hilb interpretation of the strata in Theorem \ref{thm:isom_hilb}.

\bigskip

In Section \ref{sec:torus_action} we consider the natural action of the torus $T = (\GG_m)^m$ on $\Ter^g_{\vec{c}}$ and the limits of the action of its 1-parameter subgroups. We find that these limits are often simpler, for instance, there are always limits corresponding to decomposable singularities, that is, singularities which can be expressed as a transverse union of singularities with fewer branches.

\begin{lemma}[cf. Lemma \ref{lem:limit_decomposable}] \label{lem:limit_decomposable_intro}
Let $B$ be a $k$-point of $\Ter^g_{\vec{c}}$ where $k$ is a field and let $I \sqcup I'$ be a partition of $[m]$ into nonempty subsets. There is a one-parameter subgroup $\gamma_I$ of $T$ such that the limit of $B$ under its action $\lim_{\gamma_I}B$ is in the image of the transverse union map $\vee : \Ter^{g_I}_{\vec{c}|_I} \times \Ter^{g_{I'}}_{\vec{c}|_{I'}} \to \Ter^{g}_{\vec{c}}$ for some integer partition $g_I + g_{I'} = g$ depending on $B$.
\end{lemma}

A slightly more precise version appears as Lemma~\ref{lem:limit_decomposable}.

There is a natural stratification of $\Ter^g_{\vec{c}}$ related to the torus action according to how much of the subalgebra $B$ belongs to each degree. (Indeed, it is a Bia{\l}ynicki-Birula decomposition, albeit of a singular scheme.)
The stratification by numerical monoids of \cite[Section 3]{ishii_moduli_subrings} is a special case.

\begin{definition}
Let $B$ be a $k$-point of $\Ter^g_{\vec{c}}$. Denote by $\mm$ the maximal ideal of $A^+_{\vec{c}} \otimes k$. The \emphbf{vanishing sequence} of $B$ is the sequence $(k_d)_{d \in \ZZ}$ given by
\[
  k_d = \dim_k(B \cap \mm^d / B \cap \mm^{d+1}).
\]

The \emphbf{stratum with vanishing sequence} $(k_d)$, $\mathcal{Z}_{(k_d)}$, is the locally closed subscheme of $\Ter^g_{\vec{c}}$ on which $\mathscr{B} \cap \mm^d / \mathscr{B} \cap \mm^{d+1}$ is locally free of rank $k_d$ for each $d$.
\end{definition}

By taking several 1-parameter limits, we find that any vanishing sequence is witnessed by a transverse union of monoidal singularities, leading us to our next result. In order to state it, we first recall some basic definitions related to numerical monoids.
\begin{definition}
A \emphbf{numerical monoid} is a submonoid $M$ of $\NN$ such that $\NN - M$ is a finite set.
The smallest integer $c(M)$ such that $c(M) + \NN \subseteq M$ is called the \emphbf{conductor} of $M$. The size $g(M)$ of the set $\NN - M$ is called the \emphbf{genus} of $M$.
\end{definition}

\begin{theorem} \label{thm:fixed_points_are_monoids_intro} (= Theorem \ref{thm:fixed_points_are_monoids}) Let $k$ be a field. \hfill
\begin{enumerate}
  \item The $T$-fixed $k$-points of $\Ter^g_{\vec{c}}$ are in bijection with the tuples of numerical monoids $(M_1, \ldots, M_m)$ such that $c(M_i) \leq c_i$ for each $i$ and $\sum_i g(M_i) = g$.
  \item Each orbit-closure of $T$'s action contains at least one $T$-invariant point.
  \item There is a $k$-point of $\mathcal{Z}_{(k_d)}$ if and only if there is such a tuple of
  numerical monoids $M_1, \ldots, M_m$ such that
  \[
    k_d = \# \{ i \mid d \in M_i \text{ and } d < c_i \}
  \]
  for all $d \geq 1$.
\end{enumerate}
\end{theorem}

\bigskip

Territories are known to be connected by a theorem of Ishii \cite{ishii_moduli_subrings}, which was used in \cite{bozlee_connectedness} to show connectedness of $\UU_{g,n}$.
In Section \ref{sec:aut_action} we use the action of the connected component of the identity of the automorphism group of $A^+_{\vec{c}}$ to prove that all irreducible components of $\Ter^g_{\vec{c}}$ contain a point corresponding to a smoothable singularity.

\begin{theorem} \label{thm:orbits_contain_smoothable_intro}[cf. Theorem \ref{thm:aut_orbit_contains_partition_sing}]
Let $k$ be a field of characteristic 0 and let $G_{\vec{c}}$ be the group of algebra automorphisms of $A_{\vec{c}}^+$. The closure of any orbit of $G^{\circ}_{\vec{c}} \times \Spec k$ acting on $\Ter^{g}_{\vec{c}} \times \Spec k$ contains a point corresponding to a smoothable singularity.
\end{theorem}

Our Main Theorem \ref{thm:close_connectedness_intro} follows. Remark \ref{rem:why_char_0} explains why our approach is limited to characteristic 0: in positive characteristic, there are subalgebras whose orbit closures do not contain a subalgebra of the desired kind.

\subsection{Acknowledgments}
The author is grateful to Leo Herr and Dave Swinarski for many conversations related to this work. The author thanks Jan Stevens for corrections to Figure \ref{fig:smoothability} and Piotr Oszer for useful comments. The author would also like to thank Christopher Guevara for permission to present some of the results of his thesis in this paper.

\subsection{Tool and computational resource disclosure}

The author used large language models for copy-editing, proofreading, and assistance with the preparation of Figure~\ref{fig:smoothability}. All mathematical content, arguments, and remaining errors are the responsibility of the author.

\section{Basic theory}
\label{sec:basic_theory}
\subsection{Definition and basic properties of territories}

A territory is a moduli scheme parametrizing the subalgebras of fixed corank in a fixed algebra $\mathscr{A}$, also of finite rank. More precisely, it has the following moduli functor.

\begin{definition} (See \cite[Definition 1]{ishii_moduli_subrings}, \cite[Section 2]{bozlee_guevara_smyth})
Let $S$ be a scheme and $\mathscr{A}$ a finite locally free sheaf of $\OO_S$-algebras of rank $n$. Given an $S$-scheme $f: T \to S$, a \emphbf{family of subalgebras of $\mathscr{A}$ of corank $g$} on $T$ is a quasi-coherent $\OO_T$-subalgebra $\mathscr{B}$ of $f^*\mathscr{A}$ such that the quotient $\OO_T$-module $f^*\mathscr{A} / \mathscr{B}$ is locally free of rank $g$.

The \emphbf{$g$-territory of $\mathscr{A}$} is the functor $\Ter^g_{\mathscr{A}} : (\Sch/S)^{op} \to \Set$ defined by:
\begin{enumerate}
    \item If $f : T \to S$ is an $S$-scheme, $\Ter^g_{\mathscr{A}}(T \to S)$ is the set of families of subalgebras of $\mathscr{A}$ of corank $g$ on $T$.
    \item If $h : T \to T'$ is a morphism of $S$-schemes, then $\Ter^g_{\mathscr{A}}(h) : \Ter^g_{\mathscr{A}}(T') \to \Ter^g_{\mathscr{A}}(T)$ is defined by taking a family of subalgebras to its pullback.
\end{enumerate}
\end{definition}

Observe that a point of $\Ter^g_{\mathscr{A}}$ is also a point of $\Grass(n - g, \mathscr{A})$. Our hypotheses ensure that territories are representable by a closed subscheme of the grassmannian.

\begin{lemma} \label{lem:ter_representability} (\cite[Theorem 2.5]{bozlee_guevara_smyth}, see also \cite[Theorem 1]{ishii_moduli_subrings})
Let $S$ be a scheme and $\mathscr{A}$ a finite locally free sheaf of $\OO_S$-algebras of rank $n$. The functor $\Ter^g_{\mathscr{A}}$ is represented by a closed $S$-subscheme of $\Grass(n-g, \mathscr{A})$, which by common abuse of notation we also call $\Ter^g_{\mathscr{A}}$. In particular $\Ter^g_\mathscr{A}$ is locally projective over $S$.
\end{lemma}

For our applications to curves, we will be primarily interested in the following territories.

\begin{definition}
Let $\vec{c} = (c_1,\ldots, c_m)$ be a tuple of positive integers. Let
\[
  A_{\vec{c}} = \prod_{i = 1}^m \ZZ[t_i]/(t_i^{c_i}).
\]
Denote the subalgebra in which ``constants are equal'' by
\[
  A^+_{\vec{c}} = \{ (f_1(t_1), \ldots, f_m(t_m)) \in \ZZ[t_i]/(t_i^{c_i}) \mid f_i(0) = f_j(0) \text{ for all }i,j \},
\]
and let
\[
  \Ter^g_{\vec{c}} \coloneqq \Ter^g_{A^+_{\vec{c}}}.
\]
\end{definition}

As explained in \cite[Section 2]{bozlee_connectedness}, the points of $\Ter^g_{\vec{c}}$ parametrize reduced curve singularities with $m$ branches, genus $g$, and conductor ``bounded by $\vec{c}$,'' together with the data of a normalization morphism. In brief, suppose $\tilde{C}$ is a smooth curve over an algebraically closed field $k$ and $Z$ is a subscheme of $\tilde{C}$ consisting of points
$q_1, \ldots, q_m$ with multiplicities $c_1, \ldots, c_m$. We may choose an isomorphism $\Spec A_{\vec{c}} \otimes k \cong Z$. Then
for any $k$-point of $\Ter^g_{\vec{c}}$, which is a subalgebra $B \subseteq A^+_{\vec{c}} \otimes k$, we may take the pushout
\[
\begin{tikzcd}
  \Spec A_{\vec{c}} \otimes k \ar[r] \ar[d] \ar[dr, phantom, "\lrcorner", very near end] & \tilde{C} \ar[d, "\nu"] \\
   \Spec B \ar[r] & C,
\end{tikzcd}
\]
where the left vertical map is induced by the inclusion $B \subseteq A_{\vec{c}} \otimes k$. The result is a curve $C$ with
normalization $\nu : \tilde{C} \to C$ and a singularity at the common image $x$ of $q_1, \ldots, q_m$. The isomorphism class of the complete local ring of $x$ is determined by $B$ (and does not depend on the surrounding curve $\tilde{C}$). This construction is extended to families in \cite[Section 4]{bozlee_guevara_smyth}.

In the other direction, given a reduced curve singularity $x \in C$, one computes the normalization $\nu : \tilde{C} \to C$ of $C$ at $x$, then takes the subring $\nu^\sharp(\mathscr{O}_C)/\Cond_{\tilde{C}/C}$ of $\mathscr{O}_{\tilde{C}}/\Cond_{\tilde{C}/C}$, where $\Cond_{\tilde{C}/C} = \Ann_{\OO_C}(\nu_*\mathscr{O}_{\tilde{C}}/\mathscr{O}_C)$ is the conductor ideal. The latter is isomorphic to $A_{\vec{c}}$ for some choice of $\vec{c}$ and $\nu^\sharp(\mathscr{O}_C)/\Cond_{\tilde{C}/C}$ factors through $A^+_{\vec{c}}$, so we get a point of $\Ter^g_{\vec{c}}$, unique up to the choice of isomorphism of $\mathscr{O}_{\tilde{C}}/\Cond_{\tilde{C}/C}$ with $A_{\vec{c}}$.

\medskip

Standard functoriality properties of grassmannians induce the following properties of territories.

\begin{lemma} \label{lem:ter_basics} (See \cite[Section 2]{bozlee_guevara_smyth})
Suppose that $\mathscr{A}$ is a locally free $\OO_S$-algebra of finite rank $n$.
\begin{enumerate}
  \item (Isomorphisms) If $\phi : \mathscr{A} \to \mathscr{A}'$ is an isomorphism of $\OO_S$-algebras, then there is an induced isomorphism of territories
  $\phi_* : \Ter^g_{\mathscr{A}} \to \Ter^g_{\mathscr{A}'}$ taking $\mathscr{B} \mapsto \phi(\mathscr{B})$.
  \item (Base change) If $f : T \to S$ is a morphism of schemes, then there is a natural isomorphism $\Ter^g_{\mathscr{A}} \times_S T \cong \Ter^g_{\mathscr{A}|_T}$.
  \item (Subalgebras) If $\mathscr{A}' \in \Ter^{g'}_{\mathscr{A}}(T \to S)$, then $\Ter^{g}_{\mathscr{A}'}$ is a closed subscheme of $\Ter^{g + g'}_{\mathscr{A}} \times_S T$ via $\mathscr{B} \mapsto \mathscr{B}$.
  \item (Quotients) If $\pi : \mathscr{A} \to \mathscr{A}'$ is a surjective morphism such that $\mathscr{A}'$ is locally free, then there is a closed immersion $\Ter^g_{\mathscr{A}'} \to \Ter^g_{\mathscr{A}}$ given by $\mathscr{B} \mapsto \pi^{-1}(\mathscr{B})$.
\end{enumerate}
\end{lemma}

If $\vec{c} = (c_1, \ldots, c_m)$ and $\vec{d} = (d_1, \ldots, d_m)$ are vectors of integers such that $c_i \leq d_i$ for all $i$,
then there is a natural quotient map $A^+_{\vec{d}} \to A^+_{\vec{c}}$. By (iv) above, this induces a closed immersion $\Ter^g_{\vec{c}} \to \Ter^g_{\vec{d}}$. This suggests that we should consider the subscheme parametrizing subalgebras only associated to $\vec{c}$ and not to smaller tuples of integers; we give it the following notation. Because of our convention that the $c_i$ are positive, we make a separate definition for the unique point of $\Ter^0_{1}$: it corresponds to a smooth point, which has conductor multiplicity 0.

\begin{definition} \label{def:ter_equals_c}
For $(g,m) \neq (0,1)$, we set $\Ter^g_{=\vec{c}}$ to be the open subscheme of $\Ter^g_{\vec{c}}$ defined by
\[
  \Ter^g_{=\vec{c}} \coloneqq \Ter^g_{\vec{c}} - (\Ter^g_{(c_1 - 1, c_2, \ldots, c_m)} \cup \Ter^g_{(c_1, c_2 - 1, \ldots, c_m)} \cup \cdots \cup \Ter^g_{(c_1,\ldots, c_m - 1)}),
\]
treating any territories whose subscript contains a 0 as empty.

For $(g,m) = (0,1)$, we set $\Ter^0_{=0}$ to be the unique point of $\Ter^0_{1}$.
\end{definition}

In \cite{bozlee_guevara_smyth}, the scheme $\Ter^g_{=\vec{c}}$ is denoted by $\Ter_{\mathcal{S}}(g, \vec{c})$ and called the ``territory of curve singularities of genus $g$ and conductances $\vec{c}$.'' The main result of \cite{bozlee_guevara_smyth} decomposes the moduli stack of equinormalized curves into fiber bundles whose fibers are infinitesimal thickenings of products of the schemes $\Ter^g_{=\vec{c}}$. The geometry of the larger scheme $\Ter^g_{\vec{c}}$ is somewhat easier to study (being a closed rather than locally closed subscheme of $\Grass(c - m + 1 - g, c - m + 1)$)
and still useful (as our connectedness and smoothability results imply) so it will be our focus here.

\medskip

There are simple bounds on the possible conductances of singularities. There is also a characterization of reduced Gorenstein curve singularities in terms of conductances. We will pay special attention to Gorenstein singularities, since they have been an important class of singularities in the study of moduli of curves, see for example \cite{smyth_mstable, battistella_mstable}. 

\begin{lemma} \label{lem:c_bounds} (\cite[Chapter VIII, Proposition 1.16]{altman_kleiman})
The scheme $\Ter^g_{=\vec{c}}$ is nonempty only if
\[
  g + m - 1 < \sum_{i = 1}^m c_i \leq 2(g + m - 1).
\]
A $k$-point of $\Ter^g_{=\vec{c}}$ corresponds to a Gorenstein curve singularity if and only if
\[
  \sum_{i = 1}^{m} c_i = 2(g + m - 1).
\]
\end{lemma}

In particular, once conductances are sufficiently large (namely $c_i \geq 2(g+m-1)$ for all $i$), the closed immersions $\Ter^g_{\vec{c}} \to \Ter^g_{\vec{d}}$ obtained by increasing conductances are only nilpotent thickenings. Given a reduced curve singularity of genus $g$ with $m$ branches, the quantity $\delta = g + m - 1$ is called its \emphbf{delta invariant}.

\subsection{Equations for territories} \label{ssec:ter_eqns}

Let us describe the equations for $\Ter^g_{\mathscr{A}}$ as a closed subscheme of a grassmannian in the case that $\mathscr{A}$ is a ``sheaf of local algebras.'' More precisely, suppose that $\mathscr{A}$ is a sheaf of locally free $\OO_S$-algebras of rank $n + 1$ and admits a decomposition $\mathscr{A} = \OO_S \cdot 1 \oplus \mm_{\mathscr{A}}$, where $\mm_{\mathscr{A}}$ is a sheaf of ideals of $\mathscr{A}$, locally free of rank $n$. (For example, we may take $A^+_{\vec{c}} = \ZZ \cdot 1 \oplus \mm$ where $\mm = (t_1, \ldots, t_m)$.) A sub-$\OO_S$-algebra of $\mathscr{A}$ must contain $\OO_S \cdot 1$, so by standard arguments, a sub-$\OO_S$-algebra $\mathscr{B}$ of $\mathscr{A}$ is
equivalent to a subsheaf $\mm_{\mathscr{B}}$ of $\mm_{\mathscr{A}}$ which is closed under multiplication. (Take $\mathscr{B} = \OO_S \cdot 1 \oplus \mm_{\mathscr{B}}$.) Then $\Ter^g_{\mathscr{A}}$ may be identified with the locus of multiplicatively closed subspaces in $G = \Grass(n - g, \mm_{\mathscr{A}})$. To compute this locus, we consider the composite morphism
\[
  \mu : \mm_{\mathscr{B}} \otimes_{\OO_G} \mm_{\mathscr{B}} \to \mm_{\mathscr{A}}|_G \to \mm_{\mathscr{A}}|_G/\mm_{\mathscr{B}}
\]
where $\mm_{\mathscr{B}}$ denotes the universal subbundle of $\mm_{\mathscr{A}}$, the first morphism is induced by the multiplication of $\mm_{\mathscr{A}}$, and the last
morphism is the quotient map. The vanishing of $\mu$ is well-defined since it is a morphism of locally free $\OO_G$-modules. The condition that $\mathscr{B}$ be closed under multiplication is precisely the vanishing of $\mu$, so $\Ter^g_{\mathscr{A}} = V(\mu)$ inside of $G$.

In particular,
\begin{lemma} \label{lem:territory_in_grass_max_ideal}
$\Ter^g_{\vec{c}}$ is a closed subscheme of $\Grass(c - m - g, \mm_{A^+_{\vec{c}}})$.
\end{lemma}

We may write equations for the territory $\Ter^g_{\mathscr{A}}$ more explicitly in the standard local coordinate charts of $G$. By passing to a Zariski open cover of $S$, we may assume $\mm_{\mathscr{A}}$ is free with basis $e_1, \ldots, e_n$. Let $[n] = \{1, \ldots, n\}$. The standard coordinate charts of $G$ associated to
the basis $e_1, \ldots, e_n$ are indexed by the subsets
$I \subseteq [n]$ of size $|I| = n - g$. Explicitly, if $I = \{ i_1 < \cdots < i_{n - g} \}$ we set $U_I = \Spec \OO_S[x_{i,j}]_{i \in I, j \in [n] - I}$,
with open immersion $U_I \to G$ given by taking a tuple $(x_{i,j})_{i \in I, j \in [n] - I}$ to the subsheaf of $\mm_{\mathscr{A}}|_{U_I}$ spanned by the $\OO_{U_I}$-basis
\begin{align*}
  f_{i_1} &= e_{i_1} + \sum_{j \in [n] - I} x_{i_1,j}e_j \\
  &\vdots \\
  f_{i_{n - g}} &= e_{i_{n - g}} + \sum_{j \in [n] - I} x_{i_{n - g},j}e_j.
\end{align*}

Now, on $U_I$, both the domain and codomain of $\mu$ trivialize: $\mm_{\mathscr{B}} \otimes \mm_{\mathscr{B}}|_{U_I}$ has basis $\{ f_\alpha \otimes f_\beta \}_{\alpha,\beta \in I}$ and $(\mm_{\mathscr{A}}|_G /\mm_{\mathscr{B}})|_{U_I}$ has basis $(\ol{e_k})_{k \in [n] - I}$. So, to compute the restriction of $\Ter^g_{\mathscr{A}} = V(\mu)$ on $U_I$:
\begin{enumerate}
  \item For each $\alpha, \beta \in I$, we compute the products $f_\alpha f_\beta$ in $\mm_{\mathscr{A}}|_{U_I}$ in terms of the basis $e_1, \ldots, e_n$:
  \[
    f_\alpha f_\beta = \sum_{k \in [n]} a_{\alpha,\beta,k}(\vec{x}) e_k,
  \]
  where each $a_{\alpha,\beta,k}(\vec{x}) \in \OO_S[x_{i,j}]_{i \in I, j \in [n] - I}$ is a polynomial of degree at most two.
  \item Next, for each $\alpha, \beta \in I$, we subtract off $\sum_{i \in I} a_{\alpha,\beta,i}(\vec{x}) f_i$. This cancels out the coefficients of the $e_k$ with $k \in I$ without changing the image of $f_\alpha f_\beta$ modulo $\mm_{\mathscr{B}}$:
  \begin{align*}
    f_\alpha f_\beta &\equiv \sum_{k \in [n]} a_{\alpha,\beta,k}(\vec{x}) e_k - \sum_{i \in I} a_{\alpha,\beta,i}(\vec{x}) f_i \pmod{\mm_{\mathscr{B}}} \\
      &\equiv \sum_{k \in [n] - I} \left(a_{\alpha,\beta,k}(\vec{x}) - \sum_{i \in I} a_{\alpha,\beta,i}(\vec{x}) x_{i,k}\right) e_k \pmod{\mm_{\mathscr{B}}}.
  \end{align*}
  Thus,
  \[
     \mu(f_\alpha \otimes f_\beta) = \sum_{k \in [n] - I} \left(a_{\alpha,\beta,k}(\vec{x}) - \sum_{i \in I} a_{\alpha,\beta,i}(\vec{x}) x_{i,k}\right) \ol{e_k}.
  \]
  Observe that the coefficients are typically non-homogeneous polynomials of degree at most three.
  \item We conclude
  \begin{align*}
    \Ter^g_{\mathscr{A}}|_{U_I} &\cong V(\mu)|_{U_I} \\
      &= \Spec \frac{\OO_S[x_{i,j}]_{i \in I, j \in [n] - I}}{(a_{\alpha,\beta,k}(\vec{x}) - \sum_{i \in I} a_{\alpha,\beta,i}(\vec{x}) x_{i,k} : \alpha,\beta \in I, k \in [n] - I)}.
  \end{align*}
\end{enumerate}

We find that $\Ter^g_{\vec{c}}$ may have many irreducible components, these irreducible components may be generically
nonreduced, and they may themselves be singular.

\begin{example} (Territories may be reducible.) (See also \cite[Remark 3.14]{bozlee_guevara_smyth}.)
The chart of $\Ter^3_{(6)} \otimes \Spec \QQ$ parametrizing subalgebras with basis:
\begin{align*}
\begin{matrix}
e_1 &= &t^3 &     &+ \, a_1t &+ \,a_2t^2 &+ \,a_5t^5 \\
e_2 &= &    & t^4 &+ \, b_1t &+ \,b_2t^2 &+ \,b_5t^5
\end{matrix}
\end{align*}
has two irreducible components $Z_1$ and $Z_2$. The first irreducible component is
\[
  Z_1 = V(a_1,a_2,b_1,b_2) \cong \AA^2,
\]
parametrizing subalgebras of $\QQ[t]/t^6$ of the form
\[
  \QQ \cdot 1 + \QQ \cdot (t^3 + a_5t^5) + \QQ \cdot (t^4 + b_5t^5).
\]
The territory is generically nonreduced along $Z_1$. The second irreducible component is
\[
  Z_2 = V(a_1, b_1, b_2, a_2b_5 - 2) \cong \AA^1 \times \GG_m,
\]
parametrizing subalgebras of the form
\[
  \QQ \cdot 1 + \QQ \cdot (a_2t^2 + t^3 + a_5t^5) + \QQ \cdot \left(t^4 + \frac{2}{a_2}t^5\right).
\]
\end{example}

\begin{example}
The chart of $\Ter^2_{(2,3,3)} \otimes \Spec \QQ$ indexed by
\[
\begin{matrix}
  e_1 &= &t_1 &           &           &+ \, a_{2}t_2 &+ \, a_{3}t_3 \\
  e_2 &= &    &t_2^2 &           &+ \, b_{2}t_2 &+ \, b_{3}t_3 \\
  e_3 &= &    &           &t_3^2 &+ \, c_{2}t_2 &+ \, c_{3}t_3 
\end{matrix}
\]
has a single irreducible component of dimension 3 with a singular point at the origin. The reduction of this chart has equations
\begin{align*}
  b_{3}c_{2} - b_{2}c_{3} &= 0 \\
  b_{2}c_{2} + c_{3}^2 &= 0 \\
  a_{3}c_{2} - a_{2}c_{3} &= 0 \\
  b_{2}^2 + b_{3}c_{3} &= 0 \\
  a_{3}b_{2} - a_{2}b_{3} &= 0 \\
  a_{2}b_{2} + a_{3}c_{3} &= 0
\end{align*}
\end{example}

\subsection{Based territories}

It is often convenient to work with the additional data of a basis for $\mm_\mathscr{B}$. The moduli scheme of codimension $g$ subalgebras of $\mathscr{A}$ with such a basis will yield a smooth cover of $\Ter^g_{\mathscr{A}}$ by a quasi-affine scheme. Let us state the desired moduli functor.

\begin{definition}
Let $S$ be a scheme and $\mathscr{A}$ be a sheaf of $\OO_S$-algebras of the form $\OO_S \cdot 1 \oplus \mm_{\mathscr{A}}$, where $\mm_{\mathscr{A}}$ is a sheaf of ideals of $\mathscr{A}$, locally free of rank $n$. The \emphbf{based $g$-territory} of $\mathscr{A}$ is the functor $\BTer^g_{\mathscr{A}} : (\Sch/S)^{op} \to \Set$ defined as follows:
\begin{enumerate}
    \item For each $S$-scheme $T \to S$, $\BTer^g_{\mathscr{A}}(T \to S)$ consists of the set of ordered pairs of a family $\mathscr{B}$ of subalgebras of $\mathscr{A}|_T$ of corank $g$ and an isomorphism of $\OO_T$-modules $\OO_T^{\oplus (n - g)} \to \mm_\mathscr{B} = \mm_{\mathscr{A}} \cap \mathscr{B}$.
    \item If $g : T \to T'$ is a morphism of $S$-schemes, then $\BTer^g_{\mathscr{A}}(g) : \BTer^g_{\mathscr{A}}(T') \to \BTer^g_{\mathscr{A}}(T)$ is defined by pullback.
\end{enumerate}
\end{definition}

Recall that there is a cover
\[
  \pi : U(n-g,n) \to \Grass(n - g, n)
\]
where $U(n-g,n)$ is the scheme parametrizing $n \times (n - g)$ matrices of full rank, and $\pi$ is the morphism taking a given matrix to the subspace spanned by its columns.
When $\mm_\mathscr{A}$ is free of rank $n$, $\BTer^g_{\mathscr{A}}$ is represented by the restriction of $U(n - g,n)$ to $\Ter^g_{\mathscr{A}}$:
\[
  \BTer^g_{\mathscr{A}} \coloneqq \Ter^g_{\mathscr{A}} \times_{\Grass(n - g, n)} U(n - g, n).
\]
Thus, $\BTer^g_{\mathscr{A}}$ is a quasi-affine scheme and the natural map $\BTer^g_{\mathscr{A}} \to \Ter^g_{\mathscr{A}}$ is a left $\GL_{n-g}$-torsor. In particular, it is a smooth cover of $\Ter^g_{\mathscr{A}}$ with fiber dimension $(n-g)^2$, so any questions about $\Ter^g_{\mathscr{A}}$ which are local in the smooth topology may be studied on $\BTer^g_{\mathscr{A}}$. We will find $\BTer^g_{\vec{c}}$ conceptually useful when computing limits of one-parameter families inside of $\Ter^g_{\vec{c}}$.

\subsubsection*{Equations for based territories}

Working affine locally, we may assume $\mm_{\mathscr{A}}$ has $\OO_S$-basis $e_1, \ldots, e_n$.
Write $\AA^{n,n-g} = \Spec \ZZ[x_{i,j} : i \in [n], j \in [n-g]]$ for the scheme of $n \times (n-g)$ matrices whose entry in the $i$th row and $j$th column is $x_{i,j}$. Given a subset $I \subseteq [n]$ of size $n - g$, write $\det_I$ for the determinant of the square submatrix with rows indexed by $I$. We have $U(n-g,n) = \bigcup_{I} D(\det_I) \subseteq \AA^{n,n-g}$, where the union is over such subsets.
Let us write out equations for $\BTer^g_{\mathscr{A}}$ as a closed subscheme of $U(n - g, n)$. 

Introduce the universal column vectors
\[
  f_\alpha = \sum_{i=1}^n x_{i,\alpha} e_i \qquad \alpha \in [n-g].
\]
Using the basis of $e_1, \ldots, e_n$, the products of universal column vectors can be written uniquely as
\[
  f_\alpha f_\beta = \sum_{i=1}^{n} a_{\alpha,\beta,i}(\vec{x}) e_i
\]
where $a_{\alpha,\beta,i}(\vec{x}) \in \OO_S[x_{i,j} : i \in [n], j \in [n-g] ]$ for all $\alpha, \beta \in [n-g], i \in [n]$. Observe that each $a_{\alpha,\beta,i}(\vec{x})$ either vanishes or is homogeneous of degree 2.

Closure under multiplication is the condition
\[
  f_\alpha f_\beta \in \mathrm{Span}_{\OO_S}\{\,f_1,\dots,f_{n-g}\} \qquad \forall\,\alpha,\beta,
\]
which, in $U(n-g,n)$, is equivalent to the vanishing of all $(n-g+1)\times(n-g+1)$ minors of the augmented $n\times(n-g+1)$ matrix
\[
  M_{\alpha,\beta} = \begin{bmatrix}
      x_{1,1}  & \cdots & x_{1, n - g} & a_{\alpha, \beta, 1}(\vec{x}) \\
      \vdots & & \vdots & \vdots \\
      x_{n,1} & \cdots & x_{n, n-g} & a_{\alpha, \beta, n}(\vec{x})
  \end{bmatrix}.
\]
We conclude that $\BTer^g_{\mathscr{A}}$ is the vanishing in $\AA^{n,n-g}$ of the maximal minors
of each $M_{\alpha,\beta}$ intersected with $U(n-g,n)$.

\begin{remark}
Each minor of $M_{\alpha,\beta}$ is homogeneous of degree $n-g+2$ and the determinants $\det_I$ defining $U(n-g,n)$ are homogeneous of degree $n-g$,
so it makes sense to take the projectivization of $\BTer^g_{\mathscr{A}}$. This would be a proper scheme intermediate between $\BTer^g_{\mathscr{A}}$ and $\Ter^g_{\mathscr{A}}$.
\end{remark}

In parallel to territories, we make the following abbreviations.

\begin{definition}
Given a non-negative integer $g$ and a tuple of positive integers $\vec{c}$, let $\BTer^g_{\vec{c}} \coloneqq \BTer^g_{A^+_{\vec{c}}}$ and let $\BTer^g_{=\vec{c}} = \BTer^g_{\vec{c}} \times_{\Ter^g_{\vec{c}}} \Ter^g_{=\vec{c}}$.
\end{definition}

\begin{example}[Unibranch territories with $\dim \mm_{\mathscr{B}} = 1$] \label{ex:ter_g_g_plus_2}
Suppose $m = 1$ and $g = c - 2$. Then a basis of $\mm_{\mathscr{B}}$ consists of just one element,
\[
  a_1t + a_2t^2 + \cdots + a_{g+1}t^{g+1}.
\]
The linear independence condition simplifies to the condition that at least one $a_i$ is nonzero.

The equations for $\BTer^{g}_{g+2}$ are then the $2\times 2$ minors of
\[
  \begin{bmatrix}
    0 & a_1^2 & 2a_1a_2 & \cdots & \sum_{i=1}^{\ell-1} a_ia_{\ell-i} & \cdots & \sum_{i=1}^{g} a_ia_{g +1- i} \\
    a_1 & a_2 & a_3 & \cdots & a_\ell & \cdots & a_{g+1}
  \end{bmatrix}^T.
\]

Since $\GL_1 = \GG_m$, $\Ter^{g}_{g+2}$ is the projectivization of the corresponding variety.

Working with the equations shows that $\Ter^g_{g+2}$ may be nonreduced, for example, when $g = 2$ we have
\[
  \Ter^2_4 \cong \Proj \ZZ[a_1,a_2,a_3]/(a_1^3, 2a_1^2a_2, a_1^2a_3 - 2a_1a_2^2).
\]
Observe that $a_1 \neq 0$, but $a_1^3 = 0$.

The reduction of $\Ter^g_{g+2}$ is simply a projective space: $(\Ter^g_{g+2})_{red} \cong \PP^{\lfloor (g + 2) / 2\rfloor  - 1}$, and it is a pleasant exercise to derive this from the determinantal equations. This reduction turns out to be equal to the \emphbf{spine} of the territory, discussed below. Intuitively, the only way for $a_1t + \cdots + a_{g+1}t^{g+1}$ to be a scalar multiple of its square (as required for closure under multiplication) is for its square to be zero.
\end{example}

For the remainder of the paper, we will let $\vec{c} = (c_1,\ldots,c_m)$ be a vector of positive integers, $g$ be a non-negative integer, and any territory that we refer to will be one of the form $\Ter^g_{\vec{c}}$ unless specified otherwise.

\section{The spine}
\label{sec:spine}

In this section we will show that each territory contains a large simple piece isomorphic to a grassmannian, which we call the spine. We will use spines to bound the dimensions of territories and give improved bounds on existence of non-smoothable singularities.

\subsection{Construction and dimension bounds}

We observe first that the territory of an algebra is equal to a
grassmannian in the following situation, since every equation for containment of products holds trivially.

\begin{lemma} \label{lem:sq_zero_territory}
Suppose $\mathscr{A} = \OO_S \oplus \mm$, where $\mm$ is a square-zero subsheaf of ideals of $\mathscr{A}$, locally free of rank $n$. Then $\Ter^\delta_{\mathscr{A}} \cong \Grass(n - \delta, \mm)$.
\end{lemma}
\begin{proof}
We may work Zariski locally so that $\mm$ possesses a basis. Then, since $\mm$ is square-zero, all of the $a_{\alpha,\beta,k}$'s in the equations of Section \ref{ssec:ter_eqns} vanish, so all of the equations hold trivially for any point of $\Grass(n - \delta, \mm)$.
\end{proof}

We take the spine of a territory to consist of those subalgebras factoring through the maximal subalgebra of $A^+_{\vec{c}}$ of this form, using Lemma \ref{lem:ter_basics} part (iii).

\begin{definition}
Let $A^{spine}_{\vec{c}}$ be the subring $\ZZ + (t_i^{\lceil c_i / 2 \rceil} : i = 1,\ldots,m)$
of $A^+_{\vec{c}}$. Let $\delta$ be its corank in $A^+_{\vec{c}}$.
The \emphbf{spine} $\Spine(g, \vec{c})$ of $\Ter^{g}_{\vec{c}}$ is the subscheme $\Ter^{g-\delta}_{A^{spine}_{\vec{c}}}$ of
$\Ter^g_{\vec{c}}$.
\end{definition}

 Consequently, we have the following.

\begin{theorem} \label{thm:spine_dim}
For any nonempty territory $\Ter^g_{\vec{c}}$ and any algebraically closed field $k$,
\[
  (c - m - g)\left(g + m - \frac{c}{2} - \frac{p}{2} \right) \leq \dim \Ter^g_{\vec{c}} \times \Spec k \leq (c - m - g) g.
\]
\end{theorem}
\begin{proof}
Observe that the maximal ideal of $A^{spine}_{\vec{c}}$ has rank $\sum_{i = 1}^m \lfloor \frac{c_i}{2} \rfloor = \frac{c - p}{2}$, where $p$ is the number of $c_i$'s that are odd.
By Lemma \ref{lem:sq_zero_territory}, the spine of $\Ter^{g}_{\vec{c}}$ is then a closed subscheme isomorphic to $\Grass\left(c - m - g, \frac{c-p}{2}\right)$. It is nonempty if and only if
\[
  0 \leq c - m - g \leq \frac{c - p}{2},
\]
or equivalently,
\[
  g + m \leq c \leq 2(g + m) - p.
\]
When nonempty, the relative dimension of the spine over $\Spec \ZZ$ is
\[
  \rdim \Spine(g, \vec{c}) = (c - m - g)\left(g + m - \frac{c}{2} - \frac{p}{2} \right).
\]
This gives a lower bound on the relative dimension of $\Ter^g_{\vec{c}}$.
On the other hand, $\Ter^g_{\vec{c}}$ is contained in $\Grass(c - m - g, \mm_{A^+_{\vec{c}}})$ in light of Lemma \ref{lem:territory_in_grass_max_ideal}, which has relative dimension $(c - m - g)g$.
\end{proof}

\begin{example} \label{ex:all_twos}
For any $g$ and $m$, $\Ter^g_{(2,\ldots,2)}$ is equal to its spine, $\Grass(m - g, m)$, which has dimension $(m - g)g$.
\end{example}

\begin{example}
 The spine of $\Ter^3_{(3,3,3,3)}$ is empty. A point of the spine must be contained in $k + (t_1^2, \ldots, t_4^2)$. But $k + (t_1^2,\ldots, t_4^2)$ has corank 4 in $A^+_{3,3,3,3}$, so it cannot contain any subalgebras of corank 3. This does not imply that $\Ter^3_{(3,3,3,3)}$ or even $\Ter^3_{=(3,3,3,3)}$ is empty: a planar quadruple point yields a point of $\Ter^3_{=(3,3,3,3)}$, as planar quadruple points are Gorenstein singularities of genus 3 with branch conductances all equal to 3.
\end{example}

\begin{example}[Spines for Gorenstein singularities]
Suppose $g \geq 1$ and $c = 2(g+m-1)$. Let $\delta = g + m - 1$.
\begin{itemize}
  \item If all branch conductances are even, then $\lfloor c_i / 2 \rfloor = c_i / 2$ for all $i$, so
\begin{align*}
  \Spine(g, \vec{c}) &\cong \Grass\left(2\delta - (\delta - m + 1) - m, \frac{2\delta}{2}\right) \\
    &\cong \Grass(\delta - 1, \delta) \\
    &\cong \PP^{g+m-2}.
\end{align*}
The locus in the spine corresponding to Gorenstein singularities is open and non\-empty, consisting of the subalgebras $B$ not containing any of $t_1^{c_1 - 1}, \ldots, t_m^{c_m - 1}$.

  \item If exactly two branch conductances are odd, then $\lfloor c_i / 2 \rfloor = c_i / 2 - 1/2$ for two indices $i$. It follows that the spine consists of a point, since
\begin{align*}
  \Spine(g, \vec{c}) &\cong \Grass\left(2\delta - (\delta - m + 1) - m, \frac{2\delta}{2} - 1\right) \\
    &\cong \Grass(\delta - 1, \delta - 1)
\end{align*}
This point parametrizes a non-Gorenstein singularity.

  \item If more than two branch conductances are odd, then the spine is empty.
\end{itemize}
\end{example}

\subsubsection{The union of spines}

We recall that for any $\vec{c}, \vec{d} \in \ZZ_{\geq 1}^m$ with $c_i \leq d_i$ for all $i = 1,\ldots,m$, there is a surjection
$A^+_{\vec{d}} \to A^+_{\vec{c}}$ which induces a closed
immersion $\Ter^g_{\vec{c}} \to \Ter^g_{\vec{d}}$. Therefore $\Spine(g, \vec{c})$ naturally embeds into $\Ter^g_{\vec{d}}$. Thus, a given territory contains many subvarieties isomorphic to grassmannians. The intersections of these spines are also isomorphic to grassmannians, and dimension counts show that the various spines typically are not contained
in each other.

\begin{lemma}
Suppose that $\vec{c}^{(1)}, \ldots, \vec{c}^{(r)}$ are vectors of conductances with $\vec{c}^{(j)} \leq \vec{c}$ for all $j$. Let
\[
  c^{min}_i = \min \{ c_i^{(j)} : j = 1, \ldots, r \} \quad \text{and} \quad c^{max}_i = \max \{ c_i^{(j)} : j = 1,\ldots, r \}.
\]
If $c^{min}_i \geq \lceil c^{max}_i / 2 \rceil$ for all $i$, then the intersection in $\Ter^g_{\vec{c}}$ of the corresponding spines satisfies
\[
 \bigcap_{j = 1}^r \Spine(g, \vec{c^{(j)}}) \cong \Grass\left(c^{min} - g - m, \sum_{i = 1}^m \left(c_i^{min} - \left\lceil \frac{c^{max}_i}{2} \right\rceil \right) \right)
\]
where $c^{min} = \sum_{i = 1}^m c^{min}_i$.

If $c^{min}_i < \lceil c^{max}_i / 2 \rceil$ for some $i$, then the intersection of spines is empty instead.
\end{lemma}

\begin{proof}
The intersection consists of all corank $g$ algebras $B$ factoring through
\[
  \ZZ + (t_1^{\lceil c_1^{max}/2 \rceil}, \ldots, t_m^{\lceil c_m^{max}/2 \rceil})
\]
and containing the ideal
$(t_1^{c_1^{min}}, \ldots, t_m^{c_m^{min}})$. A dimension count completes the proof.
\end{proof}

\begin{example}
Consider the unibranch territory $\Ter^6_{12}$. It contains:
\[
\begin{tikzcd}[cramped,sep=small]
 \Spine(6,12) \cong \PP^5 & & \\
 \Spine(6,11)\cong \PP^4 \ar[u,hook] & & \\
 \Spine(6,11) \cap \Spine(6,10) \cong \PP^3 \ar[u,hook] \ar[r,hook] & \Spine(6,10) \cong \Grass(3,5) & \\
 & \Spine(6,9) \cong \Grass(2,4) \ar[u,hook] & \\
 & \Spine(6,9) \cap \Spine(6,8) \cong \PP^2 \ar[u,hook] \ar[r,hook] & \Spine(6,8) \cong \PP^3 \\
 & & \Spine(6,7) \cong pt. \ar[u,hook]
\end{tikzcd}
\]
\end{example}

Using these dimension counts, we find a simple sufficient condition for the existence of singularities with a given genus and profile of branch conductances.

\begin{theorem}
If $g \geq 0$ and $\vec{c} = (c_1, \ldots, c_m)$ is a vector of even positive integers with sum $c$ satisfying $g + m - 1 < c \leq 2(g + m - 1)$, and $k$ is any algebraically closed field, then $\Ter^g_{=\vec{c}}(k)$ is nonempty.
That is, there exists an $m$-branch reduced curve singularity over $k$ of genus $g$ with conductances exactly $\vec{c}$. Moreover, $\Ter^g_{=\vec{c}} \times \Spec k$ has at least dimension equal to that of $\Spine(g,\vec{c}) \times \Spec k$.
\end{theorem}
\begin{proof}
If $c = g + m$, then the spine is isomorphic to $\Grass(0, \frac{c}{2})$, a single point, whose corresponding singularity has exactly conductances $\vec{c}$. For the remainder of the proof, assume that $g + m < c \leq 2(g + m - 1)$.

The claim will follow if we show that a dense open subset of $\Spine(g,\vec{c}) \times \Spec k$ factors through $\Ter^g_{=\vec{c}} \times \Spec k$.
Observe that any point of $\Spine(g,\vec{c})$ factoring through $\Ter^g_{(c_1, \ldots, c_i - 1, \ldots, c_m)}$ also belongs to $\Spine(g,(c_1, \ldots, c_i - 1, \ldots, c_m))$. Thus it suffices to show that $\Spine(g, (c_1,\ldots, c_i - 1, \ldots, c_m))\times \Spec k$
has smaller dimension than $\Spine(g,\vec{c}) \times \Spec k$ for each $i = 1, \ldots, m$.
Observe that
\[
  \dim \Spine(g,\vec{c}) \times \Spec k = (c - g - m)(g + m - c/2),
\]
which is positive since $g + m < c \leq 2(g + m -1)$. Meanwhile, for each $i = 1, \ldots, m$, since $c_i - 1$ is odd,
\begin{align*}
  \dim \Spine(g, (c_1,\ldots, c_i - 1, \ldots, c_m)) \times \Spec k &= (c - 1 - g - m)(g + m - c/2).
\end{align*}
This is non-negative and strictly smaller than $\dim \Spine(g, \vec{c})$, so the result follows.
\end{proof}

\subsection{Application to smoothability of singularities}
The first non-smoothable curve singularities were constructed by Mumford in \cite{mumford_pathologies}, inspired by an argument of \cite{iarrobino}. The idea is simple: a family of non-isomorphic smoothable curves of genus $g$ (with no nontrivial automorphisms) has dimension at most that of $\moduli_{g}$; hence, if one can construct a family of non-isomorphic curves of larger dimension, it must contain curves with non-smoothable singularities. Many subsequent results on non-smoothable singularities rely on the same ``large families'' idea \cite{stevens_nonsmoothable}.
Our notion of spine gives us an easy way to construct large families of singular curves, hence a way to construct many more non-smoothable curve singularities.

Let us recall the definitions of the moduli stacks of all reduced curves and smoothable curves.

\begin{definition} (\cite[Appendix B]{smyth_zstable}) \label{def:u_gn}
For non-negative integers $g, n$, we define $\UU_{g,n}$ (or $\UU_{g}$, if $n = 0$) to be the stack whose fiber over a scheme $S$ is the category of flat, proper, finitely presented morphisms of algebraic spaces $\pi : C \to S$ with connected, reduced, pure one-dimensional geometric fibers of arithmetic genus $g$, together with $n$ sections $p_1, \ldots, p_n : S \to C$. We call $\UU_{g,n}$ the \emphbf{moduli stack of $n$-marked reduced curves of genus $g$}.

Let $\mathcal{V}_{g,n}$ be the irreducible component of $\UU_{g,n}$ containing $\moduli_{g,n}$; this is the \emphbf{moduli stack of smoothable curves}.
\end{definition}

We now show that one can construct a family of singular curves over $\Ter^g_{\vec{c}}$.

\begin{theorem}
Suppose $(g,m) \neq (0,1)$. Then there is a morphism
\[
  \Ter^{g}_{\vec{c}} \to \UU_{g,2m} \setminus \moduli_{g,2m}
\]
injective on points and factoring through the part of the stack with trivial stabilizers.
\end{theorem}
\begin{proof}

Set $\tilde{X}$ equal to the disjoint union of $m$ copies of $\PP^1$. Mark the points at 1 and $\infty$ of each $\PP^1$.
Set $\mathcal{Z}$ equal to the closed subscheme of $\tilde{X}$ containing the 0 of the $i$th copy of $\PP^1$ with multiplicity $c_i$
for each $i$.
Identify $\mathcal{Z}$ with $\Spec A_{\vec{c}}$. Let $\mathscr{B}$ be the universal element of $\Ter^g_{\vec{c}}$. The inclusion $A^+_{\vec{c}} \to A_{\vec{c}}$ induces a morphism
$\mathcal{Z} \times \Ter^g_{\vec{c}} \to \SSpec{\Ter^g_{\vec{c}}}{\mathscr{B}}$.
Taking the pushout
\[
\begin{tikzcd}
  \mathcal{Z} \times \Ter^g_{\vec{c}} \ar[r] \ar[d] & \tilde{X} \times \Ter^g_{\vec{c}} \ar[d] \\
  \SSpec{\Ter^g_{\vec{c}}}{\mathscr{B}} \ar[r] & X
\end{tikzcd}
\]
yields a family of curves $X \to \Ter^g_{\vec{c}}$, i.e., a morphism
$\Ter^g_{\vec{c}} \to \UU_{g,2m}$ (cf. \cite[Section 4]{bozlee_guevara_smyth}). Since $(g,m) \neq (0,1)$, the image curves always have singularities, so it factors through the complement of $\moduli_{g,2m}$.

To prove injectivity, suppose that two $k$-points $x, y$ of $\Ter^g_{\vec{c}}$ map to points $C_x, C_y$ of $\UU_{g,2m}$ such that there exists an isomorphism $\phi : C_x \to C_y$. There is an induced isomorphism of normalizations $\tilde{\phi} : \tilde{C}_x \to \tilde{C}_y$
by the universal property of normalization. By construction, $\tilde{X} = \tilde{C}_x = \tilde{C}_y$, so we may regard $\tilde{\phi}$ as an automorphism of $\tilde{X}$. Since $\tilde{\phi}$ preserves the markings and the 0 on each copy of $\PP^1$, it must be the identity. Then in order for $\phi$ to be an isomorphism, it must be that
the subalgebras associated to $x$ and $y$ are the same, so $x = y$.

By similar reasoning, any point in the image has only the identity automorphism. We conclude that the morphism $\Ter^{g}_{\vec{c}} \to \UU_{g,2m} \setminus \moduli_{g,2m}$ factors through the part of $\UU_{g,2m} \setminus \moduli_{g,2m}$ with trivial stabilizers.
\end{proof}

If the resulting family of curves is large enough, then it must contain non-smoothable curves.

\begin{corollary} \label{cor:nonsmooth_exist}
Suppose $(g,m) \neq (0,1)$. Then if
\[
  (c - m - g)\left(g + m - \frac{c}{2} - \frac{p}{2} \right) \geq 3g - 3 + 2m,
\]
where $p$ is the number of odd $c_i$'s, then a dense open subset of $\Spine(g,\vec{c})$ parametrizes non-smoothable reduced curve singularities with conductances bounded by $\vec{c}$.
\end{corollary}
\begin{proof}
Consider the morphism $\phi : \Spine(g,\vec{c}) \to \UU_{g,2m} \setminus \moduli_{g,2m}$ of the previous theorem restricted to the spine. We check the (fairly technical) hypotheses of \cite[Tag 0DS6]{stacks-project} in order to show $\phi(\Spine(g,\vec{c}))$ and $\Spine(g,\vec{c})$ have the same relative dimension.
Since $\phi$ is one-to-one on geometric points and factors through the trivial-stabilizer locus, the diagonal $\Delta_\phi$ is injective on geometric points. Since $\Spine(g,\vec{c})$ is quasicompact, $\Delta_\phi$ is quasifinite, so by definition, $\phi$ is quasi-DM. Since both $\Spine(g, \vec{c})$ and $\UU_{g,2m}$ are locally finite type over $\Spec \ZZ$, they are both Jacobson, pseudo-catenary, and locally Noetherian. Moreover, $\UU_{g,2m}$ is quasi-separated by \cite[Theorem B.1]{smyth_zstable}. Since the geometric fibers of $\phi$ are 0-dimensional, we conclude by \cite[Tag 0DS6]{stacks-project} that the stack-theoretic image $\phi(\Spine(g,\vec{c}))$ has the same relative dimension as $\Spine(g,\vec{c})$ over $\ZZ$.

By Theorem \ref{thm:spine_dim}, this relative dimension is $(c - m - g)(g + m - c/2 - p/2)$, which by hypothesis is greater than or equal to $3g - 3 + 2m$.
The smoothable component \(\mathcal{V}_{g,2m}\) of $\UU_{g,2m}$ has relative dimension \(3g-3+2m\),
and \(\moduli_{g,2m}\) is a dense open substack. Hence the locus $\mathcal{V}_{g,2m} \setminus \moduli_{g,2m}$ of singular
smoothable curves has relative dimension at most \(3g-3+2m-1\).
It follows that $\mathcal{V}_{g,2m} \cap \phi(\Spine(g,\vec{c}))$ is a proper closed subset of $\phi(\Spine(g,\vec{c}))$. Therefore the open subset $\phi^{-1}(\UU_{g,2m} \setminus \mathcal{V}_{g,2m})$ of the irreducible scheme $\Spine(g,\vec{c})$ is nonempty, hence dense, and parametrizes non-smoothable singularities.
\end{proof}

Given a genus $g$ and number of branches $m$, we now vary over all possible conductances to get a sufficient condition for when there exists a
non-smoothable reduced curve singularity of type $(g,m)$.

\begin{theorem} \label{thm:smoothability}
Suppose $g \geq 0$, $m \geq 1$, and $(g,m) \neq (0,1)$.

\begin{enumerate}
  \item If $3g > m$, then there exists a reduced non-smoothable singularity of genus $g$ with $m$ branches whenever
\[
  (g + m)^2 - \beta \geq 24g - 24 + 16m
\]
where
\[
  \beta = \begin{cases}
      0 & \text{ if } g + m \equiv 0 \pmod{4} \\
      1 & \text{ if } g + m \equiv 1,3 \pmod{4} \\
      4 & \text{ if } g + m \equiv 2 \pmod{4}.
  \end{cases}
\]
  \item If $3g \leq m$, then there exists a reduced non-smoothable singularity of genus $g$ with $m$ branches whenever
  \[
    (m - g)g \geq 3g - 3 + 2m
  \]
\end{enumerate}
\end{theorem}
\begin{proof}
We first compute the maximum, over all choices of conductances $\vec{c}$, of the dimension of the spine of $\Ter^g_{\vec c}$.

Recall from Theorem \ref{thm:spine_dim} that for a choice of positive conductances $(c_1, \ldots, c_m)$ with sum $c$, the spine is nonempty if and only if $g + m \leq c \leq 2(g+m) - p$, where $p$ is the number of odd conductances. Moreover, when it is nonempty the dimension of the spine is $(c - m - g)\left(g + m - \frac{c}{2} - \frac{p}{2} \right)$. Observe that if any of the odd conductances are incremented by 1, the spine remains nonempty and its dimension does not decrease. Indeed, this replaces \((c,p)\) by \((c+1,p-1)\), preserving \(c+p\);
the second factor in the dimension formula is unchanged and the first factor
increases by \(1\). Thus for the purposes of finding the maximum spine dimension, we may assume each $c_i$ is even and $c_i \geq 2$. Then $c$ is even and greater than or equal to $2m$.

We are now maximizing $(c - m - g)(g + m - c/2)$ with $g, m$ held constant and $c$ allowed to vary
over the even integers satisfying $g+m \leq c \leq 2(g+m)$ and $2m \leq c$. Write $c = 2k$, so that we are maximizing
\[
  (2k - m - g)(g + m - k) = -2(k - (g + m)/2)(k - (g + m))
\]
with $k$ an integer satisfying $\frac{g+m}{2} \leq k \leq g + m$ and $m \leq k$. As this is quadratic in $k$, the maximum dimension is achieved at the closer of $k = \left\lfloor \frac{3}{4}(g + m) \right\rfloor$ or $k = \left\lceil \frac{3}{4}(g + m) \right\rceil$ to $k = \frac{3}{4}(g + m)$ unless this is less than $m$.

Suppose $\frac{3}{4}(g + m) > m$, or equivalently, $3g > m$. As
\[
  \left\lfloor \frac{3}{4}(g+m) \right\rfloor = \frac{3}{4}(g + m) - \begin{cases} 1/4 & g + m \equiv 3 \pmod{4} \\ 1/2 & g + m \equiv 2 \pmod{4} \\ 3/4 & g + m \equiv 1 \pmod{4} \\ 0 & g + m \equiv 0 \pmod{4} \end{cases}
\]
and
\[
  \left\lceil \frac{3}{4}(g+m) \right\rceil = \frac{3}{4}(g + m) + \begin{cases} 3/4 & g + m \equiv 3 \pmod{4} \\ 1/2 & g + m \equiv 2 \pmod{4} \\ 1/4 & g + m \equiv 1 \pmod{4} \\ 0 & g + m \equiv 0 \pmod{4} \end{cases}
\]
we have a maximum dimension of $\frac{1}{8}(g + m)^2 - \alpha$, where 
\[
  \alpha = \begin{cases}
      0 & g + m \equiv 0 \pmod{4} \\
      \frac{1}{8} & g + m \equiv 1,3 \pmod{4} \\
      \frac{1}{2} & g + m \equiv 2 \pmod{4}
  \end{cases}
\]

If $3g \leq m$, then the maximum dimension must occur when $k = m$. In this case the dimension of the spine is $(c - m - g)(g + m - c/2) = (m - g)g$. 

The result follows by Corollary \ref{cor:nonsmooth_exist}.
\end{proof}

\begin{remark} \label{rmk:nonsmoothability_literature}
We discuss the literature on non-smoothable singularities that enters into our Figure \ref{fig:smoothability} in more detail. For a more complete discussion, we direct the reader to \cite{stevens_nonsmoothable}.

In \cite{mumford_pathologies}, Mumford constructs a family of singular curves over the spine of a unibranch territory (but of course, without the language of territories) by inserting singularities into a curve $C$ with no automorphisms. He observes that for $c = 4v$ and singularity genus $g = 3v - 1$, the dimension of the spine (namely $v^2$) grows quadratically in $v$, while the dimension of the appropriate moduli space of smooth curves (namely $3g(C) + 9v - 6$) only grows linearly in $v$, so there must be non-smoothable singularities. Setting $g(C) = 3$, the minimum $v$ for which this construction shows there are non-smoothable singularities is $v = 10$, or equivalently, for singularity genus $g = 29$. The method was refined by Greuel \cite[Section 3.2]{greuel}, whose analysis implies that for $m = 1, g \geq 22$ there are non-smoothable singularities.

Denote by $L^n_m$ a singularity formed by taking $m$ lines through the origin in $\AA^n$. By \cite[Lemma 1.5]{stevens_nonsmoothable}, the general such singularity has genus
\[
  g(L^n_m) = (d - 1)m + 1 - \binom{n + d - 1}{d - 1}
\]
where $d = d(n,m)$ is the unique integer such that
\[
  \binom{n+d-2}{d-1} < m \leq \binom{n+d-1}{d}.
\]
In particular, for $d = 2$, the genus is given by $g = m - n$.

Pinkham proved in his thesis \cite[Theorem 1.11]{pinkham_thesis} that the general $L^n_m$ is not smoothable when $3 \leq g \leq n - 1$ and $n > 5 + \frac{6}{g-2}$. These always have $d = 2$, so, substituting $n = m - g$, it is equivalent to ask that $g \geq 3$, $m > g + 5 + \frac{6}{g-2}$, and $m \geq 2g + 1$.

Using a large-families argument, Greuel extends Pinkham's result \cite[Section 3.4]{greuel} to show that the general $L^n_m$ is not smoothable if
\begin{enumerate}
  \item $d = 2$ and $(m - n - 2)(n - 5) \geq 7$ or
  \item $d \geq 3$ and $m(n - 3d) + 3\binom{n + d - 1}{d-1} \geq n^2 - 1$.
\end{enumerate}
The condition in the $d = 2$ case can be rewritten in terms of $g$ and $m$ as $m \geq g + 5 + \frac{7}{g-2}$. A computer search shows that only two $L^n_m$'s satisfying Greuel's condition in the $d \geq 3$ case land in the plotted region of Figure \ref{fig:smoothability}, namely $L^{7}_{29}$ at $(g,m) = (23,29)$ and $L^7_{30}$ at $(g,m) = (25,30)$.

Pinkham showed that for $d = 2$, a particular $L^n_m$ is smoothable if and only if an associated set of points is a hyperplane section of a canonical curve \cite[Proposition 11.9]{pinkham_thesis}. Stevens (\cite[Proposition 10]{stevens_canonical_curves}, with the $g = 8$ case resolved in \cite[Section 3.2]{stevens_nonsmoothable}) used this to show that the general $L^n_m$ with $d = 2$ is non-smoothable if and only if 
\begin{enumerate}
    \item $m > g + 5 + \frac{6}{g - 2}$ or
    \item $m > g + 5$ and $g$ is either 4 or 6.
\end{enumerate} 

By classifying singularities with low genus, Stevens showed that all singularities with $g \leq 2$ are smoothable \cite[Lemma 4-4]{stevens_la_rabida} and the only non-smoothable singularities with $g = 3$ are joins of a non-smoothable $L^{m-3}_m$ with an ordinary $k$-fold point. Applying the result of the previous paragraph, all singularities with $g = 3$ and $m \leq 14$ are smoothable.

The only other known general method by which one can show singularities are non-smoothable is Buchweitz's criterion, introduced in \cite{buchweitz}, but more easily accessible in \cite[Section 2]{stevens_nonsmoothable}. This criterion has been used to identify non-smoothable monomial singularities. In order to complete Figure \ref{fig:smoothability}, a computer search was performed which identified the following
semigroups with non-smoothable associated singularities.

\begin{center}
\begin{tabular}{|c|c|} \hline
genus & semigroup satisfying Buchweitz's criterion \\ \hline
$g \leq 15$ & none \\ \hline
16 & $\langle 13,14,15,16,17,18,21,22,24 \rangle$\\ \hline
17 & $\langle 14,15,16,17,18,19,20,23,24,26 \rangle$ \\ \hline
18 & $\langle 15,16,17,18,19,20,21,24,26,27,28 \rangle$\\ \hline
19 & $\langle 16,17,18,19,20,21,22,23,26,28,29,30 \rangle$ \\ \hline
20 & $\langle 17,18,19,20,21,22,23,24,27,29,30,31,32 \rangle$\\ \hline
21 & $\langle 18,19,20,21,22,23,24,25,26,29,31,32,33,34 \rangle$\\ \hline
\end{tabular}
\end{center}

To our knowledge, these cover all known $(g,m)$ for which there was constructed a non-smoothable singularity with genus $g$ and $m$ branches prior to our result. 
\end{remark}

\section{Restriction, contraction, and join of branches}
\label{sec:functoriality}

In this section we define several maps between locally closed subschemes of various territories corresponding to the operations on singularities of join (or transverse union), contraction of branches, and restriction to branches. These morphisms express how multibranch territories ($m > 1$) relate to unibranch territories ($m = 1$).

We begin with the operation of transverse union. The following notation will be useful.

\begin{definition}
Given a list of positive integers $\vec{c} = (c_1, \ldots, c_m)$ and a subset $I = \{ i_1 < i_2 < \cdots < i_{r} \} \subseteq \{1, \ldots, m \}$, denote by $\vec{c}|_I$ the list $(c_{i_1}, \ldots, c_{i_{r}})$.

Let $\vec{c} = (c_1,\ldots, c_m)$ be a vector of positive integers. Write $\mm_{\vec{c}}$ for the maximal ideal of $A^+_{\vec{c}}$. Observe
\[
  \mm_{\vec{c}} = \bigoplus_{i = 1}^m \mm_{(c_i)} = \bigoplus_{i=1}^m \ZZ \left\{ t_i, \ldots, t_i^{c_i - 1} \right\}.
\]
If $I \subseteq \{ 1, \ldots, m \}$ is a subset of the indices, we write $\pi_I : \mm_{\vec{c}} \to \mm_{\vec{c}|_I}$ for the natural projection and $j_I : \mm_{\vec{c}|_I} \to \mm_{\vec{c}}$ for the natural
inclusion.
\end{definition}

\begin{definition}
Let $\vec{c} = (c_1, \ldots, c_m)$ be a list of positive integers.
For any set partition $I_1, \ldots, I_k$ of $\{1, \ldots, m \}$ and non-negative integer partition $g_1 + \cdots + g_k = g$ we define the \emphbf{join} operation to be the morphism
\[
  \bigvee : \Ter^{g_1}_{\vec{c}|_{I_1}} \times \cdots \times \Ter^{g_k}_{\vec{c}|_{I_k}} \to \Ter^g_{\vec{c}}
\]
given on $S$-points by
\[
  (\mathscr{B}_1, \ldots, \mathscr{B}_k) \mapsto \OO_S \cdot 1 \oplus j_{I_1}(\mm_{\mathscr{B}_1}) \oplus \cdots \oplus j_{I_k}(\mm_{\mathscr{B}_k}).
\]
\end{definition}

\begin{lemma}
$\bigvee$ is a closed immersion.
\end{lemma}

\begin{proof}
Proper monomorphisms are closed immersions \cite[Tag 04XV]{stacks-project}. Since the source and target are proper, the map is also proper. The map is a monomorphism since, for each $S$, the map on $S$-points $(\mathscr{B}_1, \ldots, \mathscr{B}_k) \mapsto \OO_S \cdot 1 \oplus j_{I_1}(\mm_{\mathscr{B}_1}) \oplus \cdots \oplus j_{I_k}(\mm_{\mathscr{B}_k})$ is injective.
\end{proof}

\medskip

Next, we consider contraction and restriction of branches. Since the genus of the singularities obtained by performing contractions or restrictions may vary, we first need to restrict to appropriate domains.

\begin{definition}
Let $\vec{c} = (c_1,\ldots, c_m)$ be a vector of positive integers with sum $c$. Let $I \subseteq \{ 1, \ldots, m \}$ be a subset of the indices and let $I' = \{1, \ldots, m\} - I$ be its complement, such that both $I$ and $I'$ are nonempty. Let $g$ be a non-negative integer and $g_I + g_{I'} = g$ be a partition of $g$ into non-negative integers. We define $\Gamma^{g_I, g_{I'}}_{I}$ to be the locally closed subscheme of $\Grass(c - m - g, \mm_{\vec{c}})$ on which the corank $g$ subspace $V$ of $\mm_{\vec{c}}$ has intersection $V \cap \mm_{\vec{c}|_I} = j_I^{-1}(V)$ of corank $g_I$ in $\mm_{\vec{c}|_{I}}$, or equivalently, on which $V$ has projection $\pi_{I'}(V)$ of corank $g_{I'}$ in $\mm_{\vec{c}|_{I'}}$. (More precisely, take the flattening stratum on which $\mm_{\vec{c}|_{I}} / j_I^{-1}(V)$ is locally free of rank $g_I$.) Then we set
\[
  \Ter^{g_I,g_{I'}}_{I} \coloneqq \Gamma^{g_I, g_{I'}}_{I} \cap \Ter^{g}_{\vec{c}}
\]
where the intersection is taken in $\Grass(c - m - g, \mm_{\vec{c}})$.
\end{definition}

\begin{remark} \label{rem:contraction_strata}
The spaces $\Gamma^{0, g}_I, \Gamma^{1, g-1}_I, \ldots, \Gamma^{g,0}_I$ form a locally closed stratification of $\Grass(c- g -m, \mm_{\vec{c}})$ with $\Gamma^{i, g-i}_I$ contained in the closure of $\Gamma^{j,g-j}_I$ (when nonempty) if and only if $j \geq i$. It follows that $\Ter^{0,g}_I, \ldots, \Ter^{g,0}_I$ is a decomposition of $\Ter^g_{\vec{c}}$ into disjoint locally closed subschemes, and for each $j$, the closure of $\Ter^{j, g-j}_{I}$ in $\Ter^g_{\vec{c}}$ is contained in the union $\bigcup_{i \leq j} \Ter^{i,g-i}_I$. Moreover, the first nonempty subspace among $\Ter^{0,g}_I, \ldots, \Ter^{g,0}_I$ is closed in $\Ter^g_{\vec{c}}$, and the last nonempty subspace is open.
\end{remark}

\begin{definition}
Let $\vec{c}$, $I$, $I'$, $g = g_I + g_{I'}$ be as above. We define the \emphbf{restriction to $I'$-branches} morphism by
\begin{align*}
  \rho_{I'} : \Ter^{g_I, g_{I'}}_{I} &\to \Ter^{g_{I'}}_{\vec{c}|_{I'}} \\
    (\OO_S \cdot 1 \oplus \mm_{\mathscr{B}}) &\mapsto \OO_S \cdot 1 \oplus \pi_{I'}(\mm_{\mathscr{B}}).
\end{align*}
\end{definition}

Under the correspondence with curve singularities, this corresponds to the function taking a normalized curve singularity to the union of its branches labeled by $I'$. In fact, if $\mathscr{B} \in \Ter^{g_I,g_{I'}}_I(S)$, then $\SSpec{S}{\rho_{I'}(\mathscr{B})}$ admits a natural closed immersion into $\SSpec{S}{\mathscr{B}}$, induced by $\pi_{I'}$. The associated ideal is $\mm_{\mathscr{B}} \cap j_I(\mm_{\vec{c}|_I})|_S$. This closed immersion fits into a diagram
\[
\begin{tikzcd}
    \SSpec{S}{\rho_{I'}(\mathscr{B})} \ar[r] \ar[d] & \SSpec{S}{\mathscr{B}} \ar[d] \\
    S \times \Spec A_{\vec{c}|_{I'}} \ar[r] & S \times \Spec A_{\vec{c}}
\end{tikzcd}
\]
where the vertical arrows are induced by the inclusions and horizontal arrows are closed immersions induced by $\pi_{I'}$.

\begin{definition}
Let $\vec{c}$, $I$, $I'$, $g = g_I + g_{I'}$ be as above. We define the \emphbf{contraction to $I$-branches} (or \emphbf{contraction of $I'$-branches}) morphism by
\begin{align*}
  \kappa_I : \Ter^{g_I, g_{I'}}_{I} &\to \Ter^{g_I}_{\vec{c}|_{I}} \\
   (\OO_S \cdot 1 \oplus \mm_{\mathscr{B}}) &\mapsto \OO_S \cdot 1 \oplus j_I^{-1}(\mm_{\mathscr{B}}).
\end{align*}
\end{definition}

Under the correspondence with curve singularities, this corresponds to contracting each of the branches labeled by $I'$ to the singular point. More precisely, suppose $k$ is an algebraically closed field, let $A = k[t_1,\ldots,t_m]/(t_it_j \mid 1 \leq i < j \leq m)$, and let $A_I$ be the $k$-subalgebra of $A$ generated by $t_i$ for $i \in I$. Suppose $\Spec A \to X$ is the seminormalization of a reduced curve singularity $x \in X$. Let $\Spec A \to \Spec A_I$ be the morphism induced by the inclusion: it collapses the branches labeled by $I'$.
Then we may form a new curve singularity by pushout:
\[
\begin{tikzcd}
  \Spec A \ar[d] \ar[r] \ar[dr, phantom, "\lrcorner", very near end] & \Spec A_I \ar[d] \\
X \ar[r, "\pi"] & X_I \\
\end{tikzcd}
\]
If the singularity $x \in X$ corresponds to a point $\mathscr{B} \in \Ter^{g_I,g_{I'}}_I(k)$, then the singularity $\pi(x)$ corresponds to $\kappa_I(\mathscr{B}) \in \Ter^{g_I}_{\vec{c}|_I}$.

\begin{example} \label{ex:tacnodes_restriction_contraction}
Consider $\Ter^1_{(2,2)}$. By Example \ref{ex:all_twos}, it is isomorphic to $\PP^1$, with a point $[a : b] \in \PP^1(S)$ corresponding to the subalgebra $\OO_S \oplus \OO_S(at_1 + bt_2)$ of $A^+_{(2,2)} \otimes \OO_S$. In particular, the point $[0:1]$ corresponds to a cusp on branch 1 glued transversely to a smooth branch 2, the point $[1:0]$ corresponds to a smooth branch 1 glued transversely to a cusp on branch 2, and the remaining points $[a : b]$ with $a, b \neq 0$ correspond to tacnodes.

Set $I = \{1 \}$ and $I' = \{ 2 \}$. Then $\Gamma^{1,0}_I = \{ [ a : b] \in \PP^1 \mid b \neq 0 \}$ and $\Gamma^{0,1}_{I} = \{ [1 : 0] \}$ is its complement. The restriction and contraction maps are as follows:
\begin{itemize}
  \item The restriction $\rho_{I'} : \Gamma^{1,0}_I \to \Ter^0_{(2)}$ to the second branch takes both the tacnodes and the cusp glued to a smooth branch to a smooth branch.
  \item The map $\kappa_I : \Gamma^{1,0}_I \to \Ter^1_{(2)}$ contracting the second branch takes both the tacnodes and the cusp glued to a smooth branch to a cusp. We may think that contracting the second branch of a tacnode forces us to crimp the tangent vector of the first branch to a cusp.
  \item The map $\rho_{I'} : \Gamma^{0,1}_I \to \Ter^1_{(2)}$ restricting to the second branch takes the smooth branch glued to a cusp to the branch with the cusp.
  \item The map $\kappa_I : \Gamma^{0,1}_I \to \Ter^0_{(2)}$ contracting the second branch takes the smooth branch glued to a cusp to a smooth branch.
\end{itemize}
\end{example}

Our three maps are related in the following way, which is immediate from the definitions.

\begin{lemma}
The map $\bigvee : \Ter^{g_I}_{\vec{c}|_I} \times \Ter^{g_{I'}}_{\vec{c}|_{I'}} \to \Ter^g_{\vec{c}}$ factors through $\Ter^{g_I, g_{I'}}_{I}$ and is a section of the morphism $(\kappa_I, \rho_{I'}) : \Ter^{g_I, g_{I'}}_{I} \to \Ter^{g_I}_{\vec{c}|_I} \times \Ter^{g_{I'}}_{\vec{c}|_{I'}}$. In particular, both $\rho_{I'}$ and $\kappa_I$ are surjective and $\Ter^{g_I, g_{I'}}_{I}$ is nonempty whenever $\Ter^{g_I}_{\vec{c}|_I} \times \Ter^{g_{I'}}_{\vec{c}|_{I'}}$ is.
\end{lemma}

\medskip

Let us consider how the operations of restriction and contraction interact with the multiplicities of the conductor ideal. Since the codomain of restriction is $\Ter^{g_{I'}}_{\vec{c}|_{I'}}$, the branch conductances of a restriction are less than or equal to the initial conductances. The following example shows that they may strictly decrease.

\begin{example}
Let $B = k \oplus k(t_1^2 + t_2^2) \subseteq A^+_{(3,3)}$. We have $B \in \Ter^3_{=(3,3)}$. Then the restrictions of $B$ to each branch, namely $\rho_1(B) = k \oplus kt_1^2 \subseteq A^+_{3}$ and $\rho_2(B) = k \oplus kt_2^2 \subseteq A^+_{3}$, are both singularities with conductance 2.
\end{example}

However, for a contraction, the conductances are preserved.

\begin{proposition}
Let $\vec{c}$, $I$, $I'$, $g = g_I + g_{I'}$ be as above. Then
\[
  \kappa_I(\Ter^g_{=\vec{c}} \cap \Ter^{g_I, g_{I'}}_{I}) \subseteq \Ter^{g_I}_{=\vec{c}|_I}.
\]
\end{proposition}
\begin{proof}
Since $\Ter^{g_I}_{=\vec{c}|_I}$ is the open complement of the lower conductance locus in $\Ter^{g_I}_{\vec{c}|_I}$, it suffices to check the claim
on field-valued points. Let
\[
  B = k \cdot 1 \oplus \mm_B \in (\Ter^g_{=\vec{c}} \cap \Ter^{g_I, g_{I'}}_{I})(k)
\]
for some field $k$. Since $B \in \Ter^g_{=\vec{c}}$, we have $t_i^{c_i - 1} \not\in \mm_{B}$ for each $i \in I$. Then $t_i^{c_i - 1} \not\in j_I^{-1}(\mm_B)$ for each $i \in I$, so $\kappa_I(B) \in \Ter^{g_I}_{=\vec{c}|_I}(k)$.
\end{proof}

A consequence is that the Gorenstein property is preserved under contracting a branch if and only if a simple numerical criterion holds.

\begin{corollary} \label{cor:gorenstein_contractions}
With notation as above, suppose $x$ is a geometric point of $\Ter^g_{=\vec{c}} \cap \Ter^{g_I,g_{I'}}_{I}$ and $c = 2(g + m - 1)$, so that $x$ corresponds to a Gorenstein singularity. Write $c_{I} = \sum_{i \in I} c_i$ and write $c_{I'} = \sum_{i \in I'} c_i$. Then:
\begin{enumerate}
  \item We have $c_{I'} \geq 2g_{I'} + 2|I'|$, with equality if and only if $\kappa_{I}(x)$ corresponds to a Gorenstein singularity.
  \item If $c_i = 2$ or $c_i = 3$, then the $i$th branch of the singularity corresponding to $x$ is smooth.
\end{enumerate}
\end{corollary}
\begin{proof}
By Lemma \ref{lem:c_bounds}, $c_I \leq 2g_I + 2|I| - 2$, with equality if and only if $\kappa_I(x)$ corresponds to a Gorenstein singularity. Subtracting this from $c = 2g + 2m - 2$, we have
\[
  c_{I'} = c - c_I \geq 2g + 2m - 2 - (2g_I + 2|I| - 2) = 2g_{I'} + 2|I'|,
\]
with equality if and only if $\kappa_I(x)$ corresponds to a Gorenstein singularity. This proves (i).

For (ii), suppose $c_i = 2$ or $c_i = 3$. Set $I' = \{i\}$. Since $c_i \leq 3$ and $c_i \geq 2g_{I'} + 2$, it follows that $g_{I'} = 0$, i.e., the singularity obtained by restricting to the $i$th branch has genus 0. Therefore the $i$th branch is smooth.
\end{proof}

In particular, contracting a branch of a Gorenstein singularity with odd branch conductance never results in a Gorenstein singularity. On the other hand, if a Gorenstein singularity possesses a branch of conductance 2, then that branch must be smooth, and the contraction of this branch is necessarily Gorenstein of the same genus.
This structure is apparent in the classifications of Gorenstein singularities considered in \cite[Appendix A]{smyth_mstable}, \cite[Section 2]{battistella_mstable}, and \cite{battistella_genus3_sings}. In each case there are infinite families of Gorenstein singularities obtained by gluing in smooth branches with conductance 2; contracting smooth branches of conductance 2 results in Gorenstein singularities of the same genus with fewer branches.

\begin{example}
The generic elliptic Gorenstein 4-fold point restricts to an ordinary triple point on its first three branches. Ordinary triple points are not Gorenstein, so restriction of Gorenstein singularities to branches need not result in Gorenstein singularities.
\end{example}

\begin{example}
Let us compute $\Ter^2_{=(4,2)}$, the locus of Gorenstein singularities in $\Ter^2_{(4,2)}$.  More concretely, it is the open locus in $\Ter^2_{(4,2)}$ where $B \subseteq A^+_{\vec{c}}$ does not contain $t_1^3$ or $t_2$.

Suppose $B \subseteq A^+_{\vec{c}} \otimes k$ is a $k$-point of $\Ter^2_{=(4,2)}$ for some algebraically closed field $k$. Set $I = \{ 1 \}, g_I = 2, g_{I'} = 0$ so that $\kappa_I : \Gamma^{2,0}_{I} \to \Ter^2_4$. Since $c_2 = 2$, Corollary \ref{cor:gorenstein_contractions} implies that all geometric points of $\Ter^2_{=(4,2)}$ belong to $\Ter^{2,0}_{I}$. Then $\kappa_I(B) \in \Ter^2_4$, so we may write
\[
  \kappa_I(B) = k \cdot 1 \oplus k (a_1t_1 + a_2t_1^2 + a_3t_1^3).
\]
where $a_1, a_2, a_3 \in k$ and at least one of $a_1, a_2, a_3$ is nonzero.
Moreover, Corollary \ref{cor:gorenstein_contractions} implies $\kappa_I(B)$ is Gorenstein, so at least one of $a_1, a_2$ is nonzero. If $a_1 \neq 0$, then $a_1t_1 + a_2t_1^2 + a_3t_1^3$ generates all of $A^+_4$, contradicting that $\kappa_I(B)$ is closed under multiplication, so we may take $a_1 = 0, a_2 = 1$.
We extend the basis of $\kappa_I(B)$ to a basis of $B$ by adding one more element, which necessarily has a nonzero coefficient of $t_2$:
\[
  B = k \cdot 1 \oplus k (t_1^2 + a_3t_1^3) \oplus k (b_1t_1 + b_2t_1^2 + b_3t_1^3 + t_2).
\]
Since all $k$-points of $\Ter^2_{=(4,2)}$ are of this form, it will be an open subset of the chart of $\Ter^2_{(4,2)}$ parametrizing algebras of the form
\[
  B = k \cdot 1 \oplus k (a_1t_1 + t_1^2 + a_3t_1^3) \oplus k (b_1t_1 + b_3t_1^3 + t_2).
\]

Referring to the equations of Section \ref{ssec:ter_eqns}, we find that the corresponding chart of $\Ter^2_{(4,2)}$ is
\[
  \Spec \frac{\ZZ[a_1,a_3,b_1,b_3]}{(a_1^3, 2a_1^2, a_1^2a_3 - 2a_1, a_1b_1^2, a_3b_1^2, a_1^2b_1, b_1 - a_1a_3b_1)}.
\]
Since the Gorenstein locus is cut out by the condition that $B$ does not contain $t_1^3$ or $t_2$, we consider inverting $b_1$ and $b_3$. Since the result of inverting $b_1$ is empty,
we conclude
\[
  \Ter^2_{=(4,2)} \cong \Spec \frac{\ZZ[a_1,a_3, b_1, b_3, b_3^{-1}]}{(a_1^3, 2a_1^2, a_1^2a_3 - 2a_1, a_1b_1^2, a_3b_1^2, a_1^2b_1, b_1 - a_1a_3b_1)}.
\]

Reasoning similarly,
\[
  \Ter^2_{=4} \cong \Spec \frac{\ZZ[a_1,a_3]}{(a_1^3, 2a_1^2, a_1^2a_3 - 2a_1)}
\]
and $\kappa_I$ is induced by the obvious inclusion of rings. The reductions of $\Ter^2_{=(4,2)}$ and $\Ter^2_{=4}$ can be computed by setting $a_1, b_1$ to zero. We find that
\[
  (\Ter^2_{=(4,2)})_{red} \cong \Spec \ZZ[a_3, b_3, b_3^{-1}] \cong \AA^1 \times \GG_m, \quad (\Ter^2_{=4})_{red} \cong \Spec \ZZ[a_3] \cong \AA^1,
\]
and $(\kappa_I)_{red} : \AA^1 \times \GG_m \to \AA^1$ is the obvious projection.
\end{example}

\section{Gluing branches of singularities} \label{sec:isom_hilb}

A related perspective on multibranch singularities studied by Guevara in his thesis \cite{chris_thesis} is to think of them in terms of gluing subschemes of their branches together.

For example, consider a two-branch singularity $x$ with normalization $\nu : \tilde{C} \to C$. Completing at $x$, $\nu$ induces a subalgebra $B \subseteq k\llbracket t_1 \rrbracket \times k\llbracket t_2 \rrbracket$. Goursat's Lemma states that a subalgebra of $k\llbracket t_1 \rrbracket \times k\llbracket t_2 \rrbracket$ is equivalent to the data of a pair of subalgebras $B_i \subseteq k\llbracket t_i \rrbracket$, a pair of ideals $I_i \subseteq B_i$, and an isomorphism $\phi : B_1/I_1 \to B_2/I_2$. (See for example \cite{goursat_1, goursat_2}. We give our own statement below in Lemma \ref{lem:goursat}.) Geometrically, $B_1$ and $B_2$ correspond to unibranch singularities obtained by restricting to each branch of the singularity $x$; the ideals correspond to closed subschemes $Z_1$, $Z_2$ of the branches; and $\phi$ induces an isomorphism of $Z_1$ with $Z_2$. The original subalgebra $B$ is recoverable as the fiber product
\[
\begin{tikzcd}
  B \ar[rr] \ar[d] & & B_2 \ar[d] \\
  B_1 \ar[r] & B_1/I_1 \ar[r, "\phi"] \ar[ul, phantom, "\ulcorner", very near end]&  B_2/I_2.
\end{tikzcd}
\]
Taking spectra, the original singularity is obtained by gluing $\Spec B_1$ to $\Spec B_2$ along $Z_1 \cong Z_2$ using the isomorphism $\phi$.

There is a natural locally closed stratification of a territory on which the discrete invariants of the data are fixed, namely the $k$-dimensions of $k\llbracket t_i \rrbracket / B_i$ and the common $k$-dimension of $B_i/I_i$. Then subschemes $Z_1, Z_2$ are points of a Hilbert scheme, while the isomorphism corresponds to a point of an Isom scheme. Guevara used this observation \cite{chris_thesis} to construct ``Isom-Hilb'' stratifications of territories and to compute parametrizations of territories with $(g,m) = (1,2), (1,3), (2,2)$.

Since his thesis has not appeared as a journal article, we give a presentation of this Isom-Hilb perspective. We provide our own statements and proofs, since territories of curve singularities were only articulated for reduced schemes at the time that his thesis was written. The key insight that Goursat's Lemma can be used to interpret the resulting strata is due to Guevara.

\subsection{Isom-Hilb strata}

We begin by defining the strata of interest.

\begin{definition} \label{def:isom_hilb_strata}
Let $\vec{c} = (c_1, \ldots, c_m)$ be a vector of positive integers with sum $c$.
Let $I_1 \subseteq \{1, \ldots, m\}$ be a nonempty proper subset and let $I_2$ be its complement.
Let $g$ be a non-negative integer and let $g = g_1 + g_2 + \gamma$ be a partition of $g$ into non-negative integers. The \emphbf{Isom-Hilb stratum} $\mathrm{IH}^{g_1,g_2,\gamma}_{I_1,I_2}$ is the locally closed subscheme
\[
  \mathrm{IH}^{g_1,g_2,\gamma}_{I_1,I_2} \coloneqq \Ter^{g_1 + \gamma, g_2}_{I_1} \cap \Ter^{g_2 + \gamma, g_1}_{I_2}
\]
of $\Ter^g_{\vec{c}}$.
\end{definition}

We may think of $\IH^{g_1,g_2,\gamma}_{I_1,I_2}$ as the subscheme of $\Ter^g_{\vec{c}}$ parametrizing singularities whose restriction to $I_1$-branches has genus $g_1$ and whose restriction to $I_2$-branches has genus $g_2$, or equivalently, as the subscheme parametrizing singularities whose contraction to $I_1$-branches has genus $g_1 + \gamma$ and whose contraction to $I_2$-branches has genus $g_2 + \gamma$. We remark that the roles of $I_1$ and $I_2$ here are symmetric, unlike $I$ and $I'$ above.

\begin{example}
Consider $\Ter^1_{(2,2)} \cong \PP^1$ (see Examples \ref{ex:all_twos} and \ref{ex:tacnodes_restriction_contraction}). This territory decomposes into Isom-Hilb strata as
\begin{align*}
  \IH^{0,0,1}_{\{1\},\{2\}} &= \{ [a : b] \in \PP^1 \mid a \neq 0, b \neq 0 \} \quad \text{(tacnodes)}\\
  \IH^{1,0,0}_{\{1\},\{2\}} &= \{ [0 : 1] \} \quad \text{(cusp joined with a smooth branch)}\\
  \IH^{0,1,0}_{\{1\},\{2\}} &= \{ [1 : 0] \} \quad \text{(smooth branch joined with a cusp)}.
\end{align*}
\end{example}

We call the integer $\gamma$ the ``gluing genus.'' (In \cite{chris_thesis}, $\gamma$ is called the ``gluing dimension'' $\tau$.) It may be interpreted as:
\begin{itemize}
  \item the missing genus from the restriction to branches (since $\gamma = g - g_1 - g_2$);
  \item the excess genus in the two contractions of branches (since $\gamma = (g_1 + \gamma) + (g_2 + \gamma) - g$);
  \item the difference in genus between the contraction and restriction to the $I_1$-branches or $I_2$-branches (since $\gamma = (g_1+\gamma) - g_1 = (g_2 + \gamma) - g_2$).
\end{itemize}
 We give a fourth interpretation: the excess multiplicity of the intersection of the $I_1$-branches with the $I_2$-branches.

\begin{lemma} \label{lem:intersection_of_restrictions}
Abbreviate $\IH^{g_1,g_2,\gamma}_{I_1,I_2}$ by $\IH$ and write $\mathscr{B}$ for its universal element. The intersection $\mathcal{Z}$ of the closed subschemes $\SSpec{\IH}{\rho_{I_1}(\mathscr{B})}$ and $\SSpec{\IH}{\rho_{I_2}(\mathscr{B})}$ of $\SSpec{\IH}{\mathscr{B}}$ is locally free of rank $\gamma + 1$ over $\IH$.
\end{lemma}

\begin{proof}
Recall that the ideal of $\SSpec{\IH}{\rho_{I_1}(\mathscr{B})}$ in $\SSpec{\IH}{\mathscr{B}}$ is given by $\mathscr{I}_1 \coloneqq \mm_{\mathscr{B}} \cap j_{I_2}(\mm_{\vec{c}|_{I_2}})|_{\IH}$. Symmetrically, the ideal of $\SSpec{\IH}{\rho_{I_2}(\mathscr{B})}$ is $\mathscr{I}_2 \coloneqq \mm_{\mathscr{B}} \cap j_{I_1}(\mm_{\vec{c}|_{I_1}})|_{\IH}$. The ideal of $\mathcal{Z}$ is then $\mathscr{I}_1 + \mathscr{I}_2$. Since $j_{I_1}(\mm_{\vec{c}|_{I_1}}) \cap j_{I_2}(\mm_{\vec{c}|_{I_2}}) = 0$, the intersection of the two ideals $\mathscr{I}_1 \cap \mathscr{I}_2$ is zero.

For each $j = 1,2$, write $c_{I_j} = \sum_{k \in I_j} c_k$. Using the sequence
\[
  0 \to \mathscr{I}_1 \to \mm_{\mathscr{B}} \to \pi_{I_1}(\mm_{\mathscr{B}}) \to 0,
\]
we conclude that $\mathscr{I}_1$ is locally free of rank
\[
  (c - m - g) - (c_{I_1} - |I_1| - g_1) = c_{I_2} - |I_2| - g_2 - \gamma.
\]
Similarly, $\mathscr{I}_2$ is locally free of rank $c_{I_1} - |I_1| - g_1 - \gamma$.

Since $\rho_{I_1}(\mathscr{B})$ and $\rho_{I_2}(\mathscr{B})$ are flat over $\IH$, formation of $\mathscr{I}_1$ and $\mathscr{I}_2$ commute with base change. Thus, for any $S \to \IH$ we have a short exact sequence
\[
 0 \to (\mathscr{I}_1 + \mathscr{I}_2) \otimes S \to \mathscr{B} \otimes S \to \OO_{\mathcal{Z}} \otimes S \to 0.
\]
Then $\OO_{\mathcal{Z}}$ is flat and finitely presented as an $\OO_{\IH}$-module, hence locally free.

The rank of the first term of the short exact sequence is 
\[
  (c_{I_2} - |I_2| - g_2 - \gamma) + (c_{I_1} - |I_1| - g_1 - \gamma) = c - m - g - \gamma
\]
and the middle term has rank $c - m  - g + 1$. It follows that $\OO_{\mathcal{Z}}$ has rank $\gamma + 1$, as claimed.
\end{proof}

\subsection{Goursat's Lemma}

Goursat's lemma is ordinarily stated as a bijection between subobjects of a product and ``Goursat tuples.'' We prefer to think of multibranch singularities not in terms of the normalization $A_{\vec{c}}$, which is a product, but instead in terms of the seminormalization $A^+_{\vec{c}}$, which is itself a subobject of a product, namely $A_{\vec{c}}$. Our goal in this subsection is to establish a Goursat-style characterization of subalgebras of $A^+_{\vec{c}}$. We will do so using a functorial enhancement of Goursat's lemma for $\OO_S$-algebras. We start by defining the relevant categories.

\begin{definition}
Let $S$ be a scheme. A \emphbf{product-subalgebra} on $S$ consists of $\OO_S$-algebras $\mathscr{A}_1, \mathscr{A}_2$ and an $\OO_S$-subalgebra $\mathscr{B} \subseteq \mathscr{A}_1 \times \mathscr{A}_2$.
A \emphbf{morphism of product-subalgebras} on $S$ from $\mathscr{B}' \subseteq \mathscr{A}_1' \times \mathscr{A}_2'$ to $\mathscr{B} \subseteq \mathscr{A}_1 \times \mathscr{A}_2$ is a commutative diagram of $\OO_S$-algebras
\[
\begin{tikzcd}
  \mathscr{B}' \ar[r,hook] \ar[d, "f_\mathscr{B}"] & \mathscr{A}_1' \times \mathscr{A}_2' \ar[d, "f_1 \times f_2"] \\
  \mathscr{B} \ar[r,hook] & \mathscr{A}_1 \times \mathscr{A}_2
\end{tikzcd}
\]
where $f_i : \mathscr{A}_i' \to \mathscr{A}_i$ for $i = 1,2$.

Denote the category of product-subalgebras on $S$ by $\mathcal{PS}(S)$.
\end{definition}

\begin{definition}
A \emphbf{Goursat datum} $\mathbf{B}$ on $S$ consists of
\begin{enumerate}
  \item $\OO_S$-algebras $\mathscr{A}_1, \mathscr{A}_2$
  \item $\OO_S$-subalgebras $\mathscr{B}_1 \subseteq \mathscr{A}_1, \mathscr{B}_2 \subseteq \mathscr{A}_2$
  \item sheaves of ideals $\mathscr{I}_1 \subseteq \mathscr{B}_1, \mathscr{I}_2 \subseteq \mathscr{B}_2$
  \item an $\OO_S$-algebra isomorphism $\eta : \mathscr{B}_1 / \mathscr{I}_1 \to \mathscr{B}_2 / \mathscr{I}_2$
\end{enumerate}
We may omit the ambient algebras $\mathscr{A}_1, \mathscr{A}_2$ if they are clear from context.

A \emphbf{morphism of Goursat data} $\mathbf{B}' \to \mathbf{B}$ consists of a pair of $\OO_S$-algebra morphisms $(f_1 : \mathscr{A}_1' \to \mathscr{A}_1, f_2 : \mathscr{A}_2' \to \mathscr{A}_2)$ such that
\begin{enumerate}
  \item $f_i(\mathscr{B}_i') \subseteq \mathscr{B}_i$ for $i = 1,2$
  \item $f_i(\mathscr{I}_i') \subseteq \mathscr{I}_i$ for $i = 1,2$
  \item the diagram of $\OO_S$-algebras
  \[
    \begin{tikzcd}
      \mathscr{B}_1' / \mathscr{I}_1' \ar[r, "\eta'"] \ar[d] & \mathscr{B}_2' / \mathscr{I}_2' \ar[d] \\
      \mathscr{B}_1 / \mathscr{I}_1 \ar[r, "\eta"] & \mathscr{B}_2 / \mathscr{I}_2
    \end{tikzcd}
  \]
  commutes, where the vertical arrows are the morphisms induced by $f_1, f_2$.
\end{enumerate}

Denote the \emphbf{category of Goursat data on $S$} by $\mathcal{GD}(S)$.
\end{definition}

Now we may state Goursat's lemma for $\OO_S$-algebras.

\begin{lemma} \label{lem:goursat}
There is an isomorphism of categories
\[
  \Phi: \mathcal{PS}(S) \to \mathcal{GD}(S), \quad \Psi: \mathcal{GD}(S) \to \mathcal{PS}(S)
\]
defined by:
\begin{itemize}
 \item ($\Phi$ on objects) Given a product-subalgebra $\mathscr{B} \subseteq \mathscr{A}_1 \times \mathscr{A}_2$, its image under $\Phi$ has
\[
  \mathscr{B}_1 = \pi_1(\mathscr{B}), \quad \mathscr{B}_2 = \pi_2(\mathscr{B}), \quad \mathscr{I}_1 = \iota_1^{-1}(\mathscr{B}), \quad \mathscr{I}_2 = \iota_2^{-1}(\mathscr{B}),
\]
where $\pi_i(\mathscr{B})$ is the sheaf image,
and $\eta : \mathscr{B}_1 / \mathscr{I}_1 \to \mathscr{B}_2 / \mathscr{I}_2$ is given locally by taking $[b_1] \mapsto [b_2]$ where $b_2 \in \mathscr{B}_2$ is any element such that $(b_1, b_2) \in \mathscr{B}$.

\item ($\Phi$ on morphisms) Given a morphism
\[
\begin{tikzcd}
  \mathscr{B}' \ar[r,hook] \ar[d, "f_\mathscr{B}"] & \mathscr{A}_1' \times \mathscr{A}_2' \ar[d, "f_1 \times f_2"] \\
  \mathscr{B} \ar[r,hook] & \mathscr{A}_1 \times \mathscr{A}_2
\end{tikzcd}
\]
the corresponding morphism of Goursat data is given by $(f_1 : \mathscr{A}_1' \to \mathscr{A}_1, f_2 : \mathscr{A}_2' \to \mathscr{A}_2)$.

 \item ($\Psi$ on objects) Given a Goursat datum $\mathbf{B} = (\mathscr{A}_1, \mathscr{A}_2, \mathscr{B}_1, \mathscr{B}_2, \mathscr{I}_1, \mathscr{I}_2, \eta)$, we set $\Psi(\mathbf{B})$ to be the $\OO_S$-subalgebra $\mathscr{B}_1 \times_{\mathscr{B}_2/\mathscr{I}_2} \mathscr{B}_2$ of $\mathscr{A}_1 \times \mathscr{A}_2$, where
\[
   (\mathscr{B}_1 \times_{\mathscr{B}_2/\mathscr{I}_2} \mathscr{B}_2)(U)= \{ (b_1, b_2) \in \mathscr{B}_1(U) \times \mathscr{B}_2(U) \mid \eta([b_1]) = [b_2] \}.
\]

 \item ($\Psi$ on morphisms) Given a morphism $\mathbf{B}' \to \mathbf{B}$ of Goursat data $(f_1 : \mathscr{A}_1' \to \mathscr{A}_1, f_2 : \mathscr{A}_2' \to \mathscr{A}_2)$, the corresponding morphism is given by the same maps, with $f_\mathscr{B}$ equal to the restriction of $f_1 \times f_2$.
\end{itemize}
\end{lemma}
\begin{proof}
($\Phi$ is well-defined on objects)
Suppose $\mathscr{B}$ is an $\OO_S$-subalgebra of $\mathscr{A}_1 \times \mathscr{A}_2$. We set
\begin{align*}
  \mathscr{B}_1 = \pi_1(\mathscr{B}), \quad \mathscr{B}_2 = \pi_2(\mathscr{B}), \quad \mathscr{I}_1 = \iota_1^{-1}(\mathscr{B}), \quad \mathscr{I}_2 = \iota_2^{-1}(\mathscr{B}).
\end{align*}
Observe that
\[
  \mathscr{I}_1(U) = \{ a_1 \in \mathscr{A}_1(U) \mid (a_1, 0) \in \mathscr{B}(U) \}.
\]
It is straightforward to check that $\mathscr{I}_1$ is an ideal in $\mathscr{B}_1$, as
\[
  \mathscr{B}_1(U) = \{ b_1 \in \mathscr{A}_1(U) \mid \text{locally there exists } b_2 \in \mathscr{A}_2 \text{ s.t. } (b_1, b_2) \in \mathscr{B} \}.
\]
Similarly, $\mathscr{I}_2$ is a sheaf of ideals in $\mathscr{B}_2$. Next, observe that
\[
  \mathscr{B}_1 \cong \mathscr{B} / (0 \times \mathscr{I}_2) \quad \text{ and } \quad \mathscr{B}_2 \cong \mathscr{B} / (\mathscr{I}_1 \times 0).
\]
Applying standard isomorphism theorems,
\[
  \frac{{\mathscr{B}}_1}{\mathscr{I}_1} \cong \frac{{\mathscr{B}} / (0 \times \mathscr{I}_2)}{(\mathscr{I}_1 \times \mathscr{I}_2) / (0 \times \mathscr{I}_2)} \cong \frac{{\mathscr{B}}}{\mathscr{I}_1 \times \mathscr{I}_2} \cong \frac{{\mathscr{B}} / (\mathscr{I}_1 \times 0)}{(\mathscr{I}_1 \times \mathscr{I}_2) / (\mathscr{I}_1 \times 0)} \cong \frac{{\mathscr{B}}_2}{\mathscr{I}_2},
\]
with the isomorphism given left-to-right locally by $[b_1] \mapsto [b_2]$ where $b_2$ is any element of $\mathscr{A}_2$ such that $(b_1,b_2) \in \mathscr{B}$.
This is the isomorphism $\eta$.

($\Psi$ is well-defined on objects)
In the other direction, given a tuple $(\mathscr{B}_1, \mathscr{B}_2, \mathscr{I}_1, \mathscr{I}_2, \eta)$, it is straightforward to check that
\[
  (\mathscr{B}_1 \times_{\mathscr{B}_2/\mathscr{I}_2} \mathscr{B}_2)(U) = \{ (b_1, b_2) \in \mathscr{B}_1(U) \times \mathscr{B}_2(U) \mid \eta([b_1]) = [b_2] \}
\]
defines an $\OO_S$-subalgebra of $\mathscr{A}_1 \times \mathscr{A}_2$. (The essential point is that $\eta$ and $\mathscr{B}_i \to \mathscr{B}_i / \mathscr{I}_i$ are $\OO_S$-algebra morphisms.)

($\Psi \circ \Phi$ is the identity on objects)
Let $\mathscr{B}$ be an $\OO_S$-subalgebra of $\mathscr{A}_1 \times \mathscr{A}_2$.
Let $(\mathscr{B}_1, \mathscr{B}_2, \mathscr{I}_1, \mathscr{I}_2, \eta)$ be as above and let $U \subseteq S$ be an open subset. Observe that $\mathscr{B}(U) \subseteq (\mathscr{B}_1 \times_{\mathscr{B}_2/\mathscr{I}_2} \mathscr{B}_2)(U)$ since $(b_1, b_2) \in \mathscr{B}(U)$ implies $\eta([b_1]) = [b_2]$. To check the other containment,
suppose $(b_1, b_2) \in \mathscr{B}_1(U) \times \mathscr{B}_2(U)$ such that $\eta([b_1]) = [b_2]$. By construction of $\eta$,
there is a cover $\{ U_j \}_{j \in J}$ of $U$ and elements $b_{2,j}' \in \mathscr{A}_2(U_j)$ such that $(b_1|_{U_j}, b_{2,j}') \in \mathscr{B}$ and $[b_{2,j}'] = [b_2|_{U_j}]$. Then there exist some
$i_{j} \in \mathscr{I}_2(U_j)$ for each $j$ such that $b_{2,j}' + i_j = b_2|_{U_j}$. Then $(0,i_{j}) \in \mathscr{B}$ and $(b_1|_{U_j}, b_{2,j}') + (0,i_{j}) = (b_1|_{U_j}, b_2|_{U_j})$ belongs to $\mathscr{B}(U_j)$ for each $j$. Since $\mathscr{B}$ is a subsheaf of $\mathscr{A}_1 \times \mathscr{A}_2$, $(b_1, b_2) \in \mathscr{B}(U)$, so the other containment holds.
Thus $\mathscr{B} = \mathscr{B}_1 \times_{\mathscr{B}_2 / \mathscr{I}_2} \mathscr{B}_2$, as desired.

($\Phi \circ \Psi$ is the identity on objects)
For the other direction, let $(\mathscr{A}_1, \mathscr{A}_2, \mathscr{B}_1, \mathscr{B}_2, \mathscr{I}_1, \mathscr{I}_2, \eta)$ be a Goursat datum, and write $\mathscr{B}$ for $\mathscr{B}_1 \times_{\mathscr{B}_2/\mathscr{I}_2} \mathscr{B}_2$.
It is clear that $\iota_1^{-1}(\mathscr{B}) = \mathscr{I}_1$, and $\iota^{-1}_2(\mathscr{B}) = \mathscr{I}_2$. Clearly $\pi_1(\mathscr{B}) \subseteq \mathscr{B}_1$. For the other containment, let $U \subseteq S$ be an open subset and $b_1 \in \mathscr{B}_1(U)$ be a section. By construction of $\eta$, there is an open cover $\{U_j\}_{j \in J}$ of $U$ and elements $b_{2,j} \in \mathscr{B}_2(U_j)$ such that $\eta([b_1|_{U_j}]) = [b_{2,j}]$ for all $j$. Then $(b_1|_{U_j}, b_{2,j}) \in \mathscr{B}(U_j)$ for each $j$, so $b_1|_{U_j} \in \pi_1(\mathscr{B})(U_j)$ for each $j$. Then since $\pi_1(\mathscr{B})$ is a sheaf, $b_1 \in \pi_1(\mathscr{B})(U)$. Thus $\pi_1(\mathscr{B}) = \mathscr{B}_1$. Similarly, $\pi_2(\mathscr{B}) = \mathscr{B}_2$. Finally, let $\eta' : \mathscr{B}_1 / \mathscr{I}_1 \to \mathscr{B}_2 / \mathscr{I}_2$ be the map defined locally by taking $[b_1] \mapsto [b_2]$ where $b_2$ is an element of $\mathscr{B}_2$ such that $(b_1, b_2) \in \mathscr{B}$. Then by definition of
$\mathscr{B}$, $[b_2] = \eta([b_1])$. We conclude that $\eta' = \eta$, so the two maps are inverses on objects.

($\Phi$ is well-defined on morphisms)
To see that $\Phi$ is well-defined on morphisms, suppose that
\[
\begin{tikzcd}
  \mathscr{B}' \ar[r] \ar[d, "f_\mathscr{B}"] & \mathscr{A}_1' \times \mathscr{A}_2' \ar[d, "f_1 \times f_2"] \\
   \mathscr{B} \ar[r] & \mathscr{A}_1 \times \mathscr{A}_2
\end{tikzcd}
\]
is a morphism of product-subalgebras. Suppose $b_1' \in \pi_1(\mathscr{B}')(U)$ for some open subset $U$ of $S$. Then there is an open cover
$\{ U_j \}_{j \in J}$ of $U$ and elements $b_{2,j}' \in \mathscr{A}_2'(U_j)$ such that $(b_1'|_{U_j}, b_{2,j}') \in \mathscr{B}'(U_j)$. Then
$(f_1(b_1'|_{U_j}), f_2(b_{2,j}')) \in \mathscr{B}(U_j)$, so $f_1(b_1'|_{U_j})$ belongs to $\pi_1(\mathscr{B})(U_j)$ for each $j$. Then
$f_1(b_1') \in \pi_1(\mathscr{B})(U)$. It follows that $f_1(\pi_1(\mathscr{B}')) \subseteq \pi_1(\mathscr{B})$.
Next, if $b_1' \in \mathscr{I}_1'(U)$, then
$(b_1', 0) \in \mathscr{B}'(U)$, so $(f_1(b_1'),0) \in \mathscr{B}(U)$, whence $f_1(b_1') \in \mathscr{I}_1(U)$, so $f_1(\mathscr{I}_1') \subseteq \mathscr{I}_1$. Similarly, $f_2(\mathscr{B}_2') \subseteq \mathscr{B}_2$
and $f_2(\mathscr{I}_2') \subseteq \mathscr{I}_2$. Finally, we check that
\[
    \begin{tikzcd}
      \mathscr{B}_1' / \mathscr{I}_1' \ar[r, "\eta'"] \ar[d] & \mathscr{B}_2' / \mathscr{I}_2' \ar[d] \\
      \mathscr{B}_1 / \mathscr{I}_1 \ar[r, "\eta"] & \mathscr{B}_2 / \mathscr{I}_2
    \end{tikzcd}
\]
commutes. If $[b_1'] \in \mathscr{B}_1' / \mathscr{I}_1'$ and $(b_1', b_2') \in \mathscr{B}'$ is a local lift to $\mathscr{B}'$, then under the upper path to the lower right corner, $[b_1'] \mapsto [f_2(b_2')]$. This agrees with the image under the other path since $(f_1(b_1'), f_2(b_2')) \in \mathscr{B}$.

($\Psi$ is well-defined on morphisms)
To see that $\Psi$ is well-defined on morphisms, suppose $(f_1 : \mathscr{A}_1' \to \mathscr{A}_1, f_2 : \mathscr{A}_2' \to \mathscr{A}_2)$ is a morphism of Goursat data.
We need to check that $(f_1 \times f_2)(\mathscr{B}_1' \times_{\mathscr{B}_2'/\mathscr{I}_2'} \mathscr{B}_2') \subseteq \mathscr{B}_1 \times_{\mathscr{B}_2/\mathscr{I}_2} \mathscr{B}_2$. In other words, we need to check that if
$(b_1', b_2') \in \mathscr{B}_1' \times \mathscr{B}_2'$ satisfy $\eta'([b_1']) = [b_2']$, then $\eta([f_1(b_1')]) = [f_2(b_2')]$. This holds by definition of Goursat data morphism.

The constructions respect composition and identities since $\Psi$ and $\Phi$ both act on morphisms by the same pair of maps $(f_1,f_2)$. It is straightforward to check that $\Psi$ and $\Phi$ are inverses of each other on morphisms.
\end{proof}

Next, considering the image in $\mathcal{GD}(S)$ of inclusions of subalgebras in $\mathcal{PS}(S)$, we get a characterization of Goursat data of nested subalgebras of a product.

\begin{lemma} \label{lem:goursat_sub}
Let $S$ be a scheme and suppose $\mathscr{A}_1, \mathscr{A}_2$ are $\OO_S$-algebras. If $\mathscr{B}$ and $\mathscr{B}'$ are subalgebras of $\mathscr{A}_1 \times \mathscr{A}_2$ with corresponding Goursat data
  \[
    (\mathscr{B}_1, \mathscr{B}_2, \mathscr{I}_1, \mathscr{I}_2, \eta) \text{ and } (\mathscr{B}_1', \mathscr{B}_2', \mathscr{I}_1', \mathscr{I}_2', \eta')
  \]
  then $\mathscr{B}' \subseteq \mathscr{B}$ if and only if the following properties hold
  \begin{enumerate}
     \item[(a)] \label{part:goursat_sub1} $\mathscr{B}_1' \subseteq \mathscr{B}_1$ and $\mathscr{B}_2' \subseteq \mathscr{B}_2$
     \item[(b)] \label{part:goursat_sub2} $\mathscr{I}_1' \subseteq \mathscr{I}_1$ and $\mathscr{I}_2' \subseteq \mathscr{I}_2$
     \item[(c)] \label{part:goursat_sub3} the diagram
     \[
    \begin{tikzcd}
      \mathscr{B}_1' / \mathscr{I}_1' \ar[r, "\eta'"] \ar[d] & \mathscr{B}_2' / \mathscr{I}_2' \ar[d] \\
      \mathscr{B}_1 / \mathscr{I}_1 \ar[r, "\eta"] & \mathscr{B}_2 / \mathscr{I}_2
    \end{tikzcd}
  \]
  commutes, where the vertical arrows are induced by the inclusions.
  \end{enumerate}
\end{lemma}
\begin{proof}
The forward direction is clear from Lemma \ref{lem:goursat}. Conversely, if the three properties hold, then $(\id_{\mathscr{A}_1}, \id_{\mathscr{A}_2})$
defines a morphism of Goursat data from $(\mathscr{B}_1', \mathscr{B}_2', \mathscr{I}_1', \mathscr{I}_2', \eta')$ to $(\mathscr{B}_1, \mathscr{B}_2, \mathscr{I}_1, \mathscr{I}_2, \eta)$. Applying $\Psi$, there is a morphism $(\mathscr{B}' \subseteq \mathscr{A}_1 \times \mathscr{A}_2) \to (\mathscr{B} \subseteq \mathscr{A}_1 \times \mathscr{A}_2)$ given by the identity map on $\mathscr{A}_1 \times \mathscr{A}_2$, whence $\mathscr{B}' \subseteq \mathscr{B}$.
\end{proof}

In the particular case of $A^+_{\vec{c}} \subseteq A^+_{\vec{c}|_{I_1}} \times A^+_{\vec{c}|_{I_2}}$, the criteria of the previous lemma simplify somewhat.

\begin{lemma} \label{lem:glued_subalgebras}
Let $\vec{c} = (c_1, \ldots, c_m)$ be a vector of positive integers.
Let $I_1 \subseteq \{1, \ldots, m\}$ be a nonempty proper subset and let $I_2$ be its complement.
Let $S$ be a scheme and $\mathscr{B}$ be an $\OO_S$-subalgebra of $A^+_{\vec{c}|_{I_1}}|_S \times A^+_{\vec{c}|_{I_2}}|_S$ with corresponding Goursat datum $(\mathscr{B}_1, \mathscr{B}_2, \mathscr{I}_1, \mathscr{I}_2, \eta)$. Write $\mm_{\mathscr{B}_j} \coloneqq \mm_{\vec{c}|_{I_j}} \cap \mathscr{B}_j$ for $j = 1,2$.

Then $\mathscr{B}$ factors through $A^+_{\vec{c}}|_S$ if and only if the following properties hold
\begin{enumerate}
  \item[(i)] $\mathscr{I}_1 \subseteq \mm_{\mathscr{B}_1}$ and $\mathscr{I}_2 \subseteq \mm_{\mathscr{B}_2}$.
  \item[(ii)] $\eta( \mm_{\mathscr{B}_1} / \mathscr{I}_1) = \mm_{\mathscr{B}_2} / \mathscr{I}_2$.
\end{enumerate}
\end{lemma}
\begin{proof}
We drop restrictions to $S$ for readability.
The Goursat datum corresponding to $A^+_{\vec{c}}$ in $A^+_{\vec{c}|_{I_1}} \times A^+_{\vec{c}|_{I_2}}$ is
\[
  (A^+_{\vec{c}|_{I_1}}, A^+_{\vec{c}|_{I_2}}, \mm_{\vec{c}|_{I_1}}, \mm_{\vec{c}|_{I_2}}, \theta)
\]
where $\theta$ is the unique algebra isomorphism $A^+_{\vec{c}|_{I_1}} / \mm_{\vec{c}|_{I_1}} \to A^+_{\vec{c}|_{I_2}} / \mm_{\vec{c}|_{I_2}}$ (unique since
both are isomorphic to $\OO_S$).

We now show that in this case properties (a), (b), (c) of Lemma \ref{lem:goursat_sub} are equivalent to (i), (ii) of this lemma.

((a) always holds) It is automatic that $\mathscr{B}_1 \subseteq A^+_{\vec{c}|_{I_1}}$ and $\mathscr{B}_2 \subseteq A^+_{\vec{c}|_{I_2}}$.

((i)=(b)) Immediate.

((b)+(c) $\Rightarrow$ (ii)) Supposing property (i)=(b) holds, consider the square
\[
\begin{tikzcd}
  \mathscr{B}_1 / \mathscr{I}_1 \ar[r,"\eta"] \ar[d] & \mathscr{B}_2 / \mathscr{I}_2 \ar[d] \\
  A^+_{\vec{c}|_{I_1}} / \mm_{\vec{c}|_{I_1}} \ar[r, "\theta"] & A^+_{\vec{c}|_{I_2}} / \mm_{\vec{c}|_{I_2}}.
\end{tikzcd}
\]
If (c) holds, the square commutes, so the kernel of both paths from the upper-left to lower-right of the square are the same, implying property (ii) of this Lemma.

((i)+(ii) $\Rightarrow$ (c)) Conversely, if property (ii) holds, then both paths have the same kernel, so the First Isomorphism Theorem implies that the two compositions differ by an $\OO_S$-algebra automorphism of $A^+_{\vec{c}|_{I_2}}/\mm_{\vec{c}|_{I_2}}$. Since the latter is isomorphic to $\OO_S$, the only such automorphism is the identity, so the square commutes, and (c) holds.
\end{proof}

\subsection{The functor of points of Isom-Hilb strata}

First we need one quick definition.

\begin{definition}
Given an $S$-point $\mathscr{B}$ of $\Ter^g_{\vec{c}}$, denote by $\sigma_{\vec{c}} : S \to \SSpec{S}{\mathscr{B}}$ the section induced by
the composition $\mathscr{B} \to \mathscr{B} / (\mm_{\vec{c}}|_S \cap \mathscr{B}) \overset{\sim}{\longrightarrow} \OO_S$. We call $\sigma_{\vec{c}}$ the \emphbf{canonical section}.
\end{definition}

We are now ready to give the modular interpretation of the Isom-Hilb strata that motivates their name. Essentially, a point of $\IH^{g_1,g_2,\gamma}_{I_1,I_2}$ is a point of $\Ter^{g_1}_{\vec{c}|_{I_1}}$, a point of $\Ter^{g_2}_{\vec{c}|_{I_2}}$, and a choice of how to glue them. The gluing must be compatible with how the canonical section of $\Spec A^+_{\vec{c}|_{I_1}}$ is glued to that of $\Spec A^+_{\vec{c}|_{I_2}}$
to form $\Spec A^+_{\vec{c}}$. The choice of gluing is a Goursat datum.

\begin{definition}
Let $\vec{c}, I_1, I_2, g_1, g_2, \gamma$ be as in Definition \ref{def:isom_hilb_strata}. The \emphbf{Isom-Hilb functor} is the functor
\[
  \mathcal{IH}^{g_1,g_2,\gamma}_{I_1,I_2} : \Sch^{op} \to \Set
\]
defined by:
\begin{enumerate}
  \item Given a scheme $S$
  we take $\mathcal{IH}^{g_1,g_2,\gamma}_{I_1,I_2}(S)$ to be the set of tuples $(\mathscr{B}_1, \mathscr{B}_2, \pi_1 : Z_1 \to S, \pi_2 : Z_2 \to S, \phi)$ where:
  \begin{itemize}
      \item $\mathscr{B}_j \in \Ter^{g_j}_{\vec{c}|_{I_j}}(S)$ for $j = 1,2$;
      \item $\pi_j : Z_j \to S$ is a flat family of closed subschemes of $\SSpec{S}{\mathscr{B}_j} \to S$ which are locally free of rank $\gamma + 1$ and contain the canonical section for $j = 1, 2$; and
      \item $\phi : Z_1 \to Z_2$ is an isomorphism over $S$ preserving the canonical section.
  \end{itemize}
  \item Given a morphism $f : S' \to S$, we take $\mathcal{IH}^{g_1,g_2,\gamma}_{I_1,I_2}(f)$ to be the obvious pullback.
\end{enumerate}
\end{definition}

Finally, we verify that each Isom-Hilb stratum is a representing scheme for the corresponding Isom-Hilb functor.

\begin{theorem} \label{thm:isom_hilb}
The scheme $\IH^{g_1,g_2,\gamma}_{I_1,I_2}$ represents the functor $\mathcal{IH}^{g_1,g_2,\gamma}_{I_1,I_2}$.
\end{theorem}
\begin{proof}
To unburden the notation, write $\IH$ for $\IH^{g_1,g_2,\gamma}_{I_1,I_2}$ and $\mathcal{IH}$ for $\mathcal{IH}^{g_1,g_2,\gamma}_{I_1,I_2}$.

Suppose first that $\mathscr{B} \in \IH(S)$. Let $(\mathscr{B}_1, \mathscr{B}_2, \mathscr{I}_1, \mathscr{I}_2, \eta)$ be the corresponding Goursat datum. Note that $\mathscr{B}_j = \pi_j(\mathscr{B})$ is the image of $\mathscr{B}$
under the restriction maps $\rho_{I_j}$. So by definition of $\IH$, we have $\mathscr{B}_j \in\Ter^{g_j}_{\vec{c}|_{I_j}}$.

Now let $Z_j = V(\mathscr{I}_j)$. By the proof of Goursat's Lemma, both $Z_1$ and $Z_2$ are isomorphic to $Z = V(\mathscr{I}_1 \times 0 + 0 \times \mathscr{I}_2)$. By Lemma \ref{lem:intersection_of_restrictions}, $Z$ is locally free of rank $\gamma + 1$.

Now let $\phi : Z_1 \to Z_2$ be the inverse of the isomorphism associated to $\eta : \mathscr{B}_1 / \mathscr{I}_1 \to \mathscr{B}_2 / \mathscr{I}_2$. By Lemma \ref{lem:glued_subalgebras}, each $Z_j$ contains the canonical section and $\phi$ takes the canonical section of $\SSpec{S}{\mathscr{B}_1}$ to the canonical section of $\SSpec{S}{\mathscr{B}_2}$.

Altogether, we have an element $(\mathscr{B}_1, \mathscr{B}_2, Z_1, Z_2, \phi)$ of $\mathcal{IH}(S)$.

In the other direction, suppose $(\mathscr{B}_1, \mathscr{B}_2, Z_1, Z_2, \phi)$ is an element of $\mathcal{IH}(S)$.
By Goursat's lemma, there is a corresponding subalgebra $\mathscr{B} = \mathscr{B}_1 \times_{\mathscr{B}_2/\mathscr{I}_2} \mathscr{B}_2$ of $A^+_{\vec{c}|_{I_1}}|_S \times A^+_{\vec{c}|_{I_2}}|_S$. By Lemma \ref{lem:glued_subalgebras}, it factors through $A^+_{\vec{c}}|_S$. By definition of $\mathcal{IH}$, its restrictions $\mathscr{B}_1 = \rho_{I_1}(\mathscr{B})$ and $\mathscr{B}_2 = \rho_{I_2}(\mathscr{B})$ have coranks $g_1, g_2$.

In order to have a point of $\IH$, it only remains to check that $\Delta = A^+_{\vec{c}}|_S / \mathscr{B}$ is locally free of rank $g = g_1 + g_2 + \gamma$.
Consider the short exact sequence
\[
  0 \to \mm_{\mathscr{B}_2} / \mathscr{I}_2 \to \mm_{\vec{c}|_{I_2}} / \mathscr{I}_2 \to \mm_{\vec{c}|_{I_2}} / \mm_{\mathscr{B}_2} \to 0 
\]
The term $\mm_{\mathscr{B}_2} / \mathscr{I}_2$ is locally free of rank $\gamma$ since $Z_2$ is locally free of rank $\gamma + 1$ and contains the canonical section. The last term is locally free of rank $g_2$. Therefore $\mm_{\vec{c}|_{I_2}} / \mathscr{I}_2$ is locally free of rank $g_2 + \gamma$. Next consider the short exact sequence
\[
  0 \to \mm_{\vec{c}|_{I_2}} / \mathscr{I}_2 \to A^+_{\vec{c}}|_S / \mathscr{B} \to A^+_{\vec{c}|_{I_1}}|_S / \mathscr{B}_1 \to 0.
\]
Since the first and last terms are locally free of ranks $g_2 + \gamma$ and $g_1$, respectively, the middle term is locally free of rank $g = g_1 + g_2 + \gamma$.
We conclude $\mathscr{B} \in \IH(S)$.

The two maps constructed are a restriction of the isomorphism of the Goursat lemma, so it is clear that they are inverse to each other. The constructions are clearly functorial.
\end{proof}

\section{The torus action, limits, and vanishing sequences}
\label{sec:torus_action}

There is an action of a torus $T$ on $A^+_{\vec{c}}$ given by scaling the coordinates. This induces an action of $T$ on $\Ter^g_{\vec{c}}$ which we will study in this section. We will focus on the limits of families of subalgebras induced by 1-parameter subgroups of $T$ and their relationship with graded subalgebras and vanishing sequences. As a consequence we will deduce a criterion for when there exist subalgebras with certain vanishing sequences, and we recover Ishii's stratification by numerical monoids as a special case.

\begin{definition}
Fix an integer $m$ and $\vec{c} \in \ZZ_{\geq 1}^m$ a vector of conductances. Let $T = (\GG_m)^m$. There is a natural action of $T$ on $A^+_{\vec{c}}$ given on $S$-points by
\[
   (\lambda_1, \ldots, \lambda_m) \cdot f(t_1,\ldots,t_m) = f(\lambda_1t_1, \ldots, \lambda_mt_m).
\]
Write $\phi_{(\lambda_1, \ldots, \lambda_m)}$ for the automorphism of $A^+_{\vec{c},S}$ given by $f \mapsto (\lambda_1,\ldots,\lambda_m) \cdot f$.
There is an induced action of $T$ on $\Grass(c - m + 1 - g, A^+_{\vec{c}})$ and $\Ter^g_{\vec{c}}$ given by
\[
  (\lambda_1, \ldots, \lambda_m) \cdot \mathscr{B} =  \phi_{(\lambda_1,\ldots, \lambda_m)}(\mathscr{B}),
\]
well-defined on $\Ter^g_{\vec{c}}$ since $\phi_{(\lambda_1,\ldots,\lambda_m)}$ is an $\OO_S$-algebra automorphism.
\end{definition}

It is straightforward to verify that the action of $T$ is equivariant with respect to the operations of Section \ref{sec:functoriality}.

\begin{proposition}
Let $\vec{c} = (c_1,\ldots, c_m)$ be a vector of positive integers with sum $c$. Let $I \subseteq \{ 1, \ldots, m \}$ be a subset of the indices and let $I' = \{1, \ldots, m\} - I$ be its complement, such that both $I$ and $I'$ are nonempty. Let $g$ be a non-negative integer and $g_I + g_{I'} = g$ be a partition of $g$ into non-negative integers.
Write $\pi_I : (\GG_m)^m \to (\GG_m)^{|I|}$ and $\pi_{I'} : (\GG_m)^m \to (\GG_m)^{|I'|}$ for the natural projections. Then:
\begin{enumerate}
  \item $\Ter^{g_I, g_{I'}}_I$ is $T$-invariant.
  \item For any $\mathscr{B} \in \Ter^{g_I, g_{I'}}_I(S)$, and $\vec{\lambda} \in T(S)$,
  \[
    \kappa_I(\vec{\lambda} \cdot \mathscr{B}) = \pi_I(\vec{\lambda}) \cdot \kappa_I(\mathscr{B})
  \]
  \item For any $\mathscr{B} \in \Ter^{g_I, g_{I'}}_I(S)$, and $\vec{\lambda} \in T(S)$,
  \[
    \rho_{I'}(\vec{\lambda} \cdot \mathscr{B}) = \pi_{I'}(\vec{\lambda}) \cdot \rho_{I'}(\mathscr{B})
  \]
  \item For any $\mathscr{B}_{I} \in \Ter^{g_I}_{\vec{c}|_I}(S), \mathscr{B}_{I'} \in \Ter^{g_{I'}}_{\vec{c}|_{I'}}(S)$, and $\vec{\lambda} \in T(S)$, we have
  \[
    \left(\pi_I(\vec{\lambda}) \cdot \mathscr{B}_I\right) \vee \left(\pi_{I'}(\vec{\lambda}) \cdot \mathscr{B}_{I'}\right) = \vec{\lambda} \cdot \left(\mathscr{B}_I \vee \mathscr{B}_{I'}\right)
  \]
\end{enumerate}
\end{proposition}

\begin{definition}
Fix an integer $m$ and $\vec{c} \in \ZZ_{\geq 1}^m$ a vector of conductances. Let $\iota_\gamma : \GG_m \to T$ be any 1-parameter subgroup and let $\gamma : \ZZ^m \to \ZZ$ be the corresponding morphism of cocharacter lattices. The \emphbf{$\gamma$-grading} on $A^+_{\vec{c}}$ is the
grading of $A^+_{\vec{c}}$ obtained by taking $t_i^j$ to have degree $\gamma(j e_i)$ where $e_i$ is the $i$th standard basis vector.

Given an integer $d$, we set $I^\gamma_d$ to be the subspace of $A^+_{\vec{c}}$ spanned by the elements of $\gamma$-degree $\geq d$. Note that $I^\gamma_d$ is an ideal of $A^+_{\vec{c}}$ for all $d > 0$, but not necessarily for $d \leq 0$.

Denote by $\mathscr{B}$ the universal subalgebra of $\mathscr{A}^+_{\vec{c}} \coloneqq A^+_{\vec{c}} \otimes \OO_{\Ter^g_{\vec{c}}}$, and denote by calligraphic $\mathcal{I}^\gamma_d$ the pullback of $I^\gamma_d$ to $\mathscr{A}^+_{\vec{c}}$.
Using Fitting ideals \cite[Tag 05P8]{stacks-project}, we let the \emphbf{$\gamma$-gap stratification} of $\Ter^g_{\vec{c}}$ be the
universal locally closed stratification of $\Ter^g_{\vec{c}}$
with respect to which the quotient modules $\mathscr{B} \cap \mathcal{I}^\gamma_d / \mathscr{B} \cap \mathcal{I}^\gamma_{d+1}$
are locally free for all integers $d$. (This is a finitary condition since only finitely many of the quotients are non-zero.)
The strata $\mathcal{Z}^\gamma_{(k_d)_{d \in \ZZ}}$ are indexed by the ranks
\[
  k_d \coloneqq \mathrm{rank} (\mathscr{B} \cap \mathcal{I}^\gamma_d / \mathscr{B} \cap \mathcal{I}^\gamma_{d+1}).
\]

If $B$ is an $S$-point of $\mathcal{Z}^\gamma_{(k_d)}$, we say $B$ has $\gamma$-\emphbf{vanishing sequence} $(k_d)_{d \in \ZZ}$ and $\gamma$-\emphbf{degree} $\sum_d k_d d$.
The $\gamma$-\emphbf{gap sequence} of $B$ is the sequence of coranks $(\ell_d)_{d \in \ZZ}$ where
\[
  \ell_d = \mathrm{rank} (\mathcal{I}^\gamma_d / \mathscr{B} \cap \mathcal{I}^\gamma_d).
\]

If $B$ is an $S$-point of $\mathcal{Z}^\gamma_{(k_d)}$ such that $B$ is
free, then an $\OO_S$-basis $\{ b_{d,1}, \ldots, b_{d, k_d} : d \in \ZZ \}$ of $B$ is said
to be $\gamma$-normal if and only if $b_{d,1}, \ldots, b_{d, k_d}$
belong to $\mathcal{I}^\gamma_d|_S$ and descend to a basis of $B \cap \mathcal{I}^\gamma_d|_S / B \cap \mathcal{I}^\gamma_{d+1}|_S$ for all $d \in \ZZ$.

For each definition of the form $\gamma$-X above, we will take X to mean $\gamma$-X where $\gamma$ is the sum map $(a_1,\ldots,a_m)\mapsto \sum_i a_i$. In particular, degree has its usual interpretation as total degree.
\end{definition}

\begin{example}
Recall that $\Ter^1_{(2,2)} \cong \PP^1$ with a point $[a : b] \in \PP^1$ corresponding to the algebra $\OO_S + \OO_S(at_1 + bt_2)$.
Let
\[
  \gamma_1(x,y) = x, \quad \gamma_2(x,y) = y
\]
be functions from $\ZZ^2 \to \ZZ$. There are only two possible vanishing sequences associated to these functions, namely:
\[
  k^{(1)}_d = \begin{cases} 2 & \text{ if }d = 0 \\
                            0 & \text{ else}
              \end{cases}
  \quad \text{ and } \quad
  k^{(2)}_d = \begin{cases} 1 & \text{ if }d \in \{ 0,1\} \\
                            0 & \text{ else.}
  \end{cases}
\]
The associated strata are
\begin{align*}
  \mathcal{Z}^{\gamma_1}_{(k_d^{(1)})} = \{ [a : b] \in \PP^1 \mid b \neq 0 \}, \quad &\mathcal{Z}^{\gamma_1}_{(k_d^{(2)})} = \{ [1 : 0] \} \\
  \mathcal{Z}^{\gamma_2}_{(k_d^{(1)})} = \{ [a : b] \in \PP^1 \mid a \neq 0 \}, \quad &\mathcal{Z}^{\gamma_2}_{(k_d^{(2)})} = \{ [0 : 1] \}.
\end{align*}
The stratification associated to the standard grading is trivial
\[
  \mathcal{Z}_{(k_d^{(2)})} = \PP^1
\]
\end{example}

\begin{remark}
The rank of $\mathcal{I}^\gamma_{d_0} \cap \mathscr{B}$ at a point of $\mathcal{Z}^{\gamma}_{(k_d)}$ is $\sum_{d \geq d_0} k_d$. Given a $\ZZ$-indexed sequence of integers $(l_d)$, let us say $(k_d) \preceq (l_d)$ if $\sum_{d \geq d_0} k_d \leq \sum_{d \geq d_0} l_d$ for all integers $d_0$. Then since ranks of coherent sheaves are upper-semicontinuous,
\[
  \overline{\mathcal{Z}^{\gamma}_{(k_d)}} \subseteq \bigcup_{(k_d) \preceq (l_d)} \mathcal{Z}^{\gamma}_{(l_d)}.
\]
\end{remark}

\begin{definition}
Let $B$ be a $k$-point of $\Ter^{g}_{\vec{c}}$ (resp. $\Grass(c - \delta, A^+_{\vec{c}})$), where $k$ is a field (and $\delta = g + m - 1$). The \emphbf{$\gamma$-limit of $B$}, $\lim_\gamma B$, is the $k$-point of $\Ter^g_{\vec{c}}$ (resp. $\Grass(c - \delta, A^+_{\vec{c}})$) obtained by taking the limit as $\lambda \to 0$ of $\lambda \cdot_{\gamma} B$. (More precisely, one views this as a family over $\GG_m$, completes to a family over $\AA^1$ using properness of $\Ter^g_{\vec{c}}$ (resp. $\Grass(c - \delta, A^+_{\vec{c}})$), then restricts to the fiber over 0.)
\end{definition}

\begin{remark}
The $\gamma$-limits and orbits are not invariants of isomorphism types of singularities. For example, both $k\llbracket t^2, t^5 \rrbracket \subseteq k\llbracket t \rrbracket$ and $k\llbracket t^2 + t^3, t^5 \rrbracket \subseteq k\llbracket t \rrbracket$ correspond to ramphoid cusps, but the point of $\Ter^2_{4}$ corresponding to the first is fixed by the $\GG_m$ action while the second is not. ($\lambda \cdot (k \oplus k(t^2 + t^3)) = k \oplus k(t^2 + \lambda t^3)$.) We should think that the $\gamma$-limits and $\GG_m$-orbits are invariants of singularities with a normalization by $\prod_{i = 1}^m k\llbracket t_i \rrbracket$.
\end{remark}

We will be able to compute $\gamma$-limits explicitly using a $\gamma$-normal basis. The following lemma assures us that we can always find one.

\begin{lemma}
If $B$ is an $S$-point of $\mathcal{Z}^\gamma_{(k_d)}$, then Zariski locally $B$ admits a $\gamma$-normal basis.
\end{lemma}
\begin{proof}
Working Zariski locally, we may assume that each of the quotients $B \cap \mathcal{I}^\gamma_{d}|_S / B \cap \mathcal{I}^\gamma_{d+1}|_S$ is free for each $d$. Set $b_{d,1}, \ldots, b_{d, k_d} \in B$ to a lift of a basis
for $B \cap \mathcal{I}^\gamma_{d}|_S / B \cap \mathcal{I}^\gamma_{d+1}|_S$ for each $d$. Since $B$ is finite dimensional, we have that $B \cap \mathcal{I}^\gamma_{d} = 0$ for $d$ sufficiently large and $B \cap \mathcal{I}^\gamma_{d} = B$ for $d$ sufficiently small. A descending induction argument on $d$ shows that the $b_{d, i}$'s form a
basis of $B$.
\end{proof}

\begin{lemma} \label{lem:gamma_limits}
For any $k$-point $B$ of $\Ter^g_{\vec{c}}$, we have
\begin{enumerate}
  \item The limit $\lim_\gamma B$ is $\gamma$-graded, invariant under the $\gamma$-weighted action of $\GG_m$, has the same $\gamma$-vanishing sequence as $B$, and is equal to $B$ if and only if $B$ is also $\gamma$-graded.
  \item If $\gamma_1, \gamma_2 : \ZZ^m \to \ZZ$ are two choices of grading, and $B$ is $\gamma_1$-graded, then $\lim_{\gamma_2} B$ is also $\gamma_1$-graded, and has the same $\gamma_1$-vanishing sequence as $B$.
  \item If $\{ b_{d,1}, \ldots, b_{d, k_d} : d \in \ZZ \}$ is a $\gamma$-normal basis of $B$, then a $\gamma$-normal basis of $\lim_\gamma B$ is given by taking the minimal degree summands of each $b_{d,j}$.
  \end{enumerate}
\end{lemma}

\begin{proof}
Item (i) follows from (iii). Let us prove (iii).

For each $d \in \ZZ$, $j = 1, \ldots, k_d$, write $b_{d,j} = \sum_{e} f_{d,j,e}$ where $f_{d,j,e}$ is the degree $e$ part of $b_{d,j}$.
Define a family over $\GG_m$ by taking $\lambda \in \GG_m$ to $\{ \lambda \cdot b_{d,1}, \ldots, \lambda \cdot b_{d,k_d} \}$.
We get the subalgebra of $k[\lambda, \lambda^{-1}] \otimes A^+_{\vec{c}}$ with $k[\lambda, \lambda^{-1}]$-basis
\[
  \{ \sum_{e} \lambda^e f_{d,1,e}, \ldots, \sum_{e} \lambda^e f_{d,k_d,e} : d \in \ZZ \}.
\]
The smallest degree $e$ in which $f_{d,j,e}$ is nonzero is $e = d$, so we may rescale to obtain a new basis for the same family of subalgebras
with only non-negative powers of $\lambda$, namely
\[
  \{ \sum_{e} \lambda^{e-d} f_{d,1,e}, \ldots, \sum_{e} \lambda^{e-d}f_{d,k_d,e} : d \in \ZZ \}.
\]
Now, plugging in $\lambda = 0$, the limit of this family as $\lambda \to 0$ in $\BTer$ is
\[
  \{ f_{d,1,d}, \ldots, f_{d,k_d,d} : d \in \ZZ \},
\]
provided that this remains linearly independent. But linear independence is clear since our original basis is $\gamma$-normal. 
Therefore the limit subalgebra is that spanned by the basis above, as claimed. Since the reduction of the basis elements modulo powers of $\mathcal{I}^{\gamma}$ is unchanged, this is also a $\gamma$-normal basis.

To prove (ii) we notice that we can choose $\gamma_2$-normal bases of each of the graded pieces of $B$, and compute the $\gamma_2$-limit graded piece by graded piece.
\end{proof}

\begin{example}
Consider the action of $\GG_m$ on the reduction of the unibranch territory $\Ter^{g}_{g+2}$.
Let $\ell = \lceil (g+2)/2 \rceil$, and $r = g - \ell + 1$.
Recall from Example \ref{ex:ter_g_g_plus_2} that $(\Ter^g_{g+2})_{red} \cong \PP^{r}$. An $S$-point $[a_\ell : \cdots : a_{g+1} ] \in \PP^r$ corresponds to the subalgebra
\[
  B = \OO_S \cdot 1 + \OO_S \cdot (a_\ell t^\ell + \cdots + a_{g+1}t^{g+1}) \subseteq \OO_S[t]/t^{g+2}.
\]
The action of $\GG_m$ is given by $\lambda \cdot [a_\ell : \cdots : a_{g+1}] = [\lambda^\ell a_\ell : \cdots : \lambda^{g+1} a_{g+1}]$.

There are only two interesting degree assignments: $\gamma(x) = x$ and its negation.
Suppose $B = k \cdot 1 + k(a_{i_1}t^{i_1} + \cdots + a_{i_2}t^{i_2})$ is a $k$-point for some field $k$ and $i_1, i_2$ are respectively the smallest and largest
indices for which $a_i$ is nonzero. We may write
\begin{align*}
  \lambda \cdot_{\gamma} B &= k \cdot 1 + k \left(t^{i_1} + \left(\frac{a_{i_1+1}}{a_{i_1}}\lambda\right)t^{i_1 +1} + \cdots + \left(\frac{a_{i_2}}{a_{i_1}}\lambda^{i_2 - i_1}\right)t^{i_2} \right) \\
  \lambda \cdot_{-\gamma} B &= k \cdot 1 + k \left(\left( \frac{a_{i_1}}{a_{i_2}}\lambda^{i_2 - i_1}\right)t^{i_1} + \left(\frac{a_{i_1+1}}{a_{i_2}}\lambda^{i_2 - i_1 - 1}\right)t^{i_1 +1} + \cdots + t^{i_2} \right) \\
\end{align*}
The corresponding limits are then
\[
  \lim_\gamma B = k \cdot 1 + k \cdot t^{i_1} \quad \text{and} \quad \lim_{-\gamma} B = k \cdot 1 + k \cdot t^{i_2}
\]

When $g = 6, c = 8$ we have the point $k\cdot 1 + k (t^4 + t^6 + t^7) \subseteq k[t]/t^8$. Its orbit closure is the cuspidal cubic, since
\[
  \lambda \cdot [1 : 0 : 1 : 1] = [\lambda^4 : 0 : \lambda^6 : \lambda^7 ] = [1 : 0 : \lambda^2 : \lambda^3 ].
\]
Thus orbit closures are not necessarily normal.
\end{example}

\begin{lemma} \label{lem:limit_decomposable}
Let $B$ be a $k$-point of $\Ter^g_{\vec{c}}$ where $k$ is a field. For each $i = 1, \ldots, m$, let $\gamma_i : \ZZ^m \to \ZZ$ be the projection onto the $i$th coordinate. If $I \subseteq \{1, \ldots, m \}$, let $\gamma_I : \ZZ^m \to \ZZ$ be the sum $\sum_{i \in I} \gamma_i$. Then $\lim_{\gamma_I}B$ is decomposable into the transverse union of an $|I|$-branch singularity and an $m-|I|$-branch singularity.
\end{lemma}

\begin{proof}
This follows from (iii) in the previous lemma. After taking the $\gamma_I$-limit, the resulting subalgebra is $\gamma_I$-graded. The positive $\gamma_I$-degree pieces live on the branches indexed by $I$, while the $\gamma_I$-degree 0 pieces live on branches indexed by the complement of $I$. This gives a transverse-union decomposition.
\end{proof}

We recall some basic definitions related to numerical monoids.
\begin{definition}
A \emphbf{numerical monoid} is a submonoid $M$ of $\NN$ such that $\NN - M$ is a finite set.
The smallest integer $c(M)$ such that $c(M) + \NN \subseteq M$ is called the \emphbf{conductor} of $M$. The size $g(M)$ of the set $\NN - M$ is called the \emphbf{genus} of $M$.
\end{definition}

If $M$ is a numerical monoid of positive genus, then $\Spec k[M]$ is a curve with a singularity of genus $g(M)$ and conductance $c(M)$.

\begin{theorem} \label{thm:fixed_points_are_monoids} Let $k$ be a field. \hfill
\begin{enumerate}
  \item The $T$-fixed $k$-points of $\Ter^g_{\vec{c}}$ are in bijection with the tuples of numerical monoids $(M_1, \ldots, M_m)$ such that $c(M_i) \leq c_i$ for each $i$ and $\sum_i g(M_i) = g$.
  \item Each orbit-closure of $T$'s action contains at least one $T$-invariant point.
  \item There is a $k$-point of $\mathcal{Z}_{(k_d)}$ if and only if there is such a tuple of
  numerical monoids $M_1, \ldots, M_m$ such that
  \[
    k_d = \# \{ i \mid d \in M_i \text{ and } d < c_i \}
  \]
  for all $d \geq 1$.
\end{enumerate}
\end{theorem}

\begin{proof}
For each $i = 1, \ldots, m$, let $\gamma_i : \ZZ^m \to \ZZ$ be the projection onto the $i$th coordinate and denote by $\gamma : \ZZ^m \to \ZZ$ the sum map, inducing the standard grading.

To see (i), suppose $B$ is a $T$-fixed $k$-point of $\Ter^g_{\vec{c}}$. Then by Lemma \ref{lem:gamma_limits}(i) it is graded with respect to all $\gamma_1, \ldots, \gamma_m$, so $\mm_B$ admits a basis by elements of the form $t_i^{d}$. Since $B$ is closed under multiplication
\[
  M_i = \{ d \in \NN \mid t_i^d \in B \} \cup \ZZ_{\geq c_i}
\]
is a numerical monoid with conductor at most $c_i$ for each $i$. It is straightforward to check that $B$ is the point of
$\Ter^g_{\vec{c}}$ associated to the join of the monoidal singularities $\Spec k[M_i]$. Since $B$ is a point of $\Ter^g_{\vec{c}}$,
we have $g = \sum_{i = 1}^m g(M_i)$.

The inverse function is given by taking a tuple $M_1, \ldots, M_m$ of numerical monoids with $i$th conductor at most $c_i$ and total genus $g$ to the subalgebra
\[
  k \cdot 1 \oplus \mathrm{Span} \{ t_i^d \mid 1 \leq i \leq m, 1 \leq d < c_i, d \in M_i \}.
\]
It is clearly graded with respect to $\gamma_1, \ldots, \gamma_m$ and therefore $T$-invariant.

To see (ii), suppose $B$ is a $k$-point of $\Ter^g_{\vec{c}}$. Then
\[
  B' = \lim_{\gamma_{m}}\cdots\lim_{\gamma_1} \lim_\gamma B
\]
is graded with respect to $\gamma_1, \ldots, \gamma_m$ by Lemma \ref{lem:gamma_limits}(i). Since $B'$ is invariant with respect to the $\GG_m$-actions associated to $\gamma_1, \ldots, \gamma_m$, it is invariant under the action of $T$.

To see (iii), observe that $B'$ has the same vanishing sequence as $B$ by Lemma \ref{lem:gamma_limits}(ii), since we started by taking the limit with respect to the standard grading.
Then there is a point of $\mathcal{Z}_{(k_d)}$ if and only if there is one of the form
\[
  k \cdot 1 \oplus \mathrm{Span} \{ t_i^d \mid 1 \leq i \leq m, 1 \leq d < c_i, d \in M_i \}
\]
for numerical monoids $M_1, \ldots, M_m$. We compute that 
the vanishing sequence is
\[
  k_d = \#\{ i \mid d \in M_i \text{ and } d < c_i \}.
\]
\end{proof}

Plugging in $m = 1$ to Theorem \ref{thm:fixed_points_are_monoids}, we obtain:

\begin{corollary} 
The unibranch territory $\Ter^g_c$ has a stratification into locally closed subschemes $\mathcal{Z}_M$ indexed by the numerical monoids $M$ of genus $g$ with conductor at most $c$.
Each stratum contains a unique $\GG_m$-invariant $k$-point, associated to the monoidal singularity $\Spec k[M]$.
\end{corollary}

\begin{example}
Consider the $k$-point $B = k + k(t_1 + t_2)$ of $\Ter^1_{(2,2)}$. Then
\[
  \lim_{\gamma_1} \lim_{\gamma_2} B = k + kt_1
\]
while
\[
  \lim_{\gamma_2} \lim_{\gamma_1} B = k + kt_2.
\]
That is, the limit, and even the $\gamma_1$- and $\gamma_2$-vanishing sequences of the limit depend on the order in which the limits are taken. Both limits are graded according to the usual grading since $B$ was already graded.
\end{example}

\section{The action of \texorpdfstring{$\Aut(A_{\vec{c}}^+)$}{Aut(A\_c+)} and close connectedness of \texorpdfstring{$\UU_{g,n}$}{U\_gn}}
\label{sec:aut_action}

Throughout this section we will let $k$ be an algebraically closed field of characteristic zero and work in $\Sch / k$. Accordingly, we will write $A^+_{\vec{c}}$ instead of $A^+_{\vec{c}} \otimes k$ and $\Ter^g_{\vec{c}}$ instead of $\Ter^g_{\vec{c}} \times \Spec k$. Fix an integer $g \geq 0$ and a vector $\vec{c} = (c_1, \ldots, c_m) \in \ZZ_{\geq 1}^m$ with sum $c$.
Write $\mm$ for the maximal ideal of $A^+_{\vec{c}}$.

\begin{definition}
 We define $G_{\vec{c}}$ by setting $G_{\vec{c}}(S)$ to be the group $\Aut(A^+_{\vec{c}})(S)$ of $\OO_S$-algebra automorphisms of $A^+_{\vec{c}} \otimes \OO_S$, for each $k$-scheme $S$.
\end{definition}

It is a general fact that automorphism groups of free $\OO_S$-algebras are group schemes.

\begin{lemma}
Let $S$ be a scheme and suppose $\mathscr{A}$ is a free $\OO_S$-algebra of finite rank $n$. Then the group functor $\Aut(\mathscr{A}) : (\Sch/S)^{op} \to \Grp$ taking $f : T \to S$ to the $\OO_T$-algebra automorphisms of $f^*\mathscr{A}$ is representable by a closed group subscheme of $\GL_n(\OO_S)$.
\end{lemma}
\begin{proof}
Let $e_1, \ldots, e_n \in \Gamma(S, \mathscr{A})$ be a basis of $\mathscr{A}$ and use this basis to identify $\mathscr{A}$ with $\OO_S^n$. Then there is an obvious inclusion $\Aut(\mathscr{A}) \to \GL_n(\OO_S)$ of functors to $\Grp$ since $\OO_T$-algebra automorphisms are in particular $\OO_T$-module automorphisms. Write $f_{i,j} = e_ie_j$ for the products of basis elements for $i = 1, \ldots, n, j = 1, \ldots, n$. Let $\phi : \mathscr{A} \to \mathscr{A}$ be an $\OO_S$-module automorphism. Then $\phi \in \Aut(\mathscr{A})(S)$ if and only if $\phi(1) = 1$ and $\phi(e_i)\phi(e_j) = \phi(f_{i,j})$ for all $i = 1, \ldots, n, j = 1, \ldots, n$. Writing these equations in terms of coefficients of basis representations, we get a finite system of equations which must be satisfied in order for $\phi \in \GL_n(S)$ to factor through $\Aut(\mathscr{A})$. These same equations pull back to the required equations for $\phi \in \GL_n(T)$ to factor through $\Aut(\mathscr{A})(T)$, so we conclude that $\Aut(\mathscr{A})$ is representable by a closed group subscheme of $\GL_n(\OO_S)$.
\end{proof}

The following lemma is helpful for identifying automorphisms of $A^+_{\vec{c}}$.

\begin{lemma}
Let $S$ be a scheme and $\phi : A_{\vec{c},S}^+ \to A_{\vec{c},S}^+$ be an $\OO_S$-algebra endomorphism such that $\phi(\mm_S) \subseteq \mm_S$. Then $\phi$ is an isomorphism if and only if the induced morphism of $\OO_S$-modules $\ol{\phi} : \mm_S / \mm_S^2 \to \mm_S / \mm_S^2$ is an isomorphism.
\end{lemma}
\begin{proof}
The assertion is local on $S$, so we may assume $S$ is affine.

Suppose $\phi$ is an isomorphism with inverse $\psi$. Since $\phi(\mm_S) \subseteq \mm_S$ and $A^+_{\vec{c},S}/\mm_S \cong \OO_S$, we may form a commutative diagram of $\OO_S$-modules
\[
\begin{tikzcd}
  0 \ar[r] & \mm_S \ar[r] \ar[d, "\phi"] & A^+_{\vec{c},S} \ar[r] \ar[d, "\phi"] & \OO_S \ar[r] \ar[d] & 0 \\
  0 \ar[r] & \mm_S \ar[r] & A^+_{\vec{c},S} \ar[r] & \OO_S \ar[r] & 0.
\end{tikzcd}
\]
where the rows are exact and the right square consists of $\OO_S$-algebra homomorphisms. The only $\OO_S$-algebra homomorphism $\OO_S \to \OO_S$ is the identity. A diagram chase then shows that $\phi : \mm_S \to \mm_S$ is a bijection, so $\phi(\mm_S) = \mm_S$. Taking powers, $\phi(\mm_S^2) = \mm_S^2$. Then $\psi(\mm_S) = \mm_S$ and $\psi(\mm_S^2) = \mm_S^2$, so there is an induced map $\ol{\psi} : \mm_S / \mm_S^2 \to \mm_S / \mm_S^2$
which is clearly inverse to $\ol{\phi}$.

Conversely, suppose $\ol{\phi}$ is an isomorphism. Then for each $i = 1, \ldots, m$ there exists $x_i \in \mm_S$ such that $\ol{\phi}(\ol{x_i}) = \ol{t_i}$. Then $\phi(x_i) = t_i + u$ for some $u \in \mm_S^2$. Taking powers, $\phi(x_i^j) - t_i^j \in \mm_S^{j+1}$ for all $j \geq 1$. We now show by induction that for each integer $r \geq 2$ there is an element $x_i^{(r)}$ such that $\phi(x_i^{(r)}) \equiv t_i \pmod{\mm_S^{r}}$. The base case is given by taking $x_i^{(2)} = x_i$. For the induction step, suppose $x_i^{(r)} \in A^+_{\vec{c},S}$ is an element such that $\phi(x_i^{(r)}) \equiv t_i \pmod{\mm_S^{r}}$.
Let $a_1, \ldots, a_m \in \Gamma(S, \OO_S)$ be the unique elements such that $\phi(x_i^{(r)}) = t_i + a_1t_1^{r} + \cdots + a_mt_m^{r} \pmod{\mm_S^{r+1}}$, taking $a_i = 0$ when $i \geq c_i$. Set $x_i^{(r+1)} = x_i^{(r)} - a_1x_1^{r} - \cdots - a_mx_m^{r}$. Then
\[
  \phi(x_i^{(r+1)}) = t_i \pmod{\mm_S^{r+1}}.
\]
This proves the inductive step. Now, taking $r$ sufficiently large, $\mm_S^r = 0$, so that $\phi(x_i^{(r)}) = t_i$ for each $i$. Then $\phi(A^+_{\vec{c},S})$ contains the generators of $A^+_{\vec{c},S}$ as an $\OO_S$-algebra, so $\phi$ is surjective. Now $\phi$ is a surjective endomorphism of finite free $\OO_S$-modules, so $\phi$ is an isomorphism.
\end{proof}

The group scheme $G_{\vec{c}}$ acts naturally on $\Ter^g_{\vec{c}}$ by
\begin{align*}
  \Aut_{\OO_S-alg}(A_{\vec{c}}^+ \otimes \OO_S) \times \Ter^g_{\vec{c}}(S) &\to \Ter^g_{\vec{c}}(S) \\
  (\phi, B) &\mapsto \phi(B).
\end{align*}

This extends the action of $T$ from the previous section. In this section, we will be especially interested in the action of a family of automorphisms $\phi_a$ that we introduce now.

\begin{definition}
Denote by $\phi_a$ the automorphism $\phi_a \in G_{\vec{c}}^\circ(\GG_m)$ given by $t_i \mapsto t_i + a^{-1}t_i^{2}$ for each $i = 1, \ldots, m$, where $a$ denotes the standard coordinate on $\GG_m$.
\end{definition}

We remark that $\phi_a$ belongs to the connected component of the identity of $G_{\vec{c}}$ since the limit as $a \to \infty$ of $\phi_a$ is the identity map. We will consider repeatedly taking the limit of the action of $\phi_a$ as $a \to 0$ in the following sense.

\begin{definition}
Let $B \in \Ter^g_{\vec{c}}(k)$ be a subalgebra. Denote by $\lim\limits_{a \to 0} \phi_a(B)$ the $k$-point of $\Ter^g_{\vec{c}}$ obtained by first taking the unique extension of $\phi_a(B|_{\GG_m}) \in \Ter^g_{\vec{c}}(\GG_m)$ to $\Ter^g_{\vec{c}}(k[a])$, then restricting to the vanishing of $a$.
\end{definition}

Intuitively, the limit $\lim\limits_{a \to 0} \phi_a(B)$ attempts to replace $t_i$ with $t_i^2$, so taking the limit will shift the subalgebra into higher and higher degrees. We verify this formally in Lemma \ref{lem:one_aut_limit} below for graded subalgebras. Once degrees are high enough, we will have a so-called partition subalgebra.

\begin{definition}
We say that a point of $\Ter^g_{\vec{c}}(k)$ of the form
\[
  k \cdot 1 + \mathrm{Span} \{ t_i^{d} \mid d_i \leq d < c_i \}
\]
for some integers $d_1, \ldots, d_m$ is a \emphbf{partition subalgebra}.
\end{definition}

We remark that partition subalgebras correspond to partition singularities, which are always smoothable. (See, for example \cite[Lemma 3.2]{bozlee_connectedness}.)

\begin{lemma} \label{lem:one_aut_limit}
Let $B \in \Ter^g_{\vec{c}}$ be a subalgebra with a monomial basis $\{1\} \cup \{ t_i^j \}_{(i,j) \in I}$ where $I \subseteq \{ (i,j) \mid 1 \leq i \leq m, 1 \leq j < c_i \}$. Then
$\lim\limits_{a \to 0} \phi_a(B)$ admits a monomial basis of the form
\[
  \{1\} \cup \{ t_i^{j + \eta(i,j)} \}_{(i,j) \in I}
\]
where $\eta(i,j) \geq 0$ for all $i,j \in I$. Moreover at least one $\eta(i,j) > 0$ unless $B$ is a partition subalgebra.
\end{lemma}

The proof is notation heavy, so we first give an example to illustrate the strategy.

\begin{example}
Consider the subalgebra $B$ of $k[t]/t^6$ with basis $1, t^3, t^4$. Then $\phi_a(B|_{\GG_m})$ is the $k[a,a^{-1}]$ algebra with basis $1, (t+a^{-1}t^2)^3, (t+a^{-1}t^2)^4$. Expanding,
\begin{align*}
  (t+a^{-1}t^2)^3 &= t^3 + 3a^{-1}t^4 + 3a^{-2}t^5 \\
  (t+a^{-1}t^2)^4 &= t^4 + 4a^{-1}t^5.
\end{align*}
Thus, with respect to the basis $1, t, t^2, t^3, t^4, t^5$ of $k[a,a^{-1},t]/t^6$, the algebra $\phi_a(B|_{\GG_m})$ is the subalgebra of $k[a,a^{-1},t]/t^6$ spanned by the rows of the matrix
\[
  \begin{bmatrix}
    1 & 0 & 0 & 0 & 0 & 0 \\
    0 & 0 & 0 & 1 & 3a^{-1} & 3a^{-2} \\
    0 & 0 & 0 & 0 &  1 & 4a^{-1}
  \end{bmatrix}.
\]
We now use row operations to find an alternate basis with only nonnegative powers of $a$ that remains a basis after substituting $a = 0$. The basis element 1 will be fixed, so we omit the first row and column. First, we rescale so that all trailing entries are 1.
\[
  \begin{bmatrix}
     0 & 0 & \frac{1}{3}a^2 & a & 1 \\
     0 & 0 & 0 &  \frac{1}{4}a & 1
  \end{bmatrix}
\]
We subtract the bottom row from the first row to cancel the trailing 1 of the first row. Observe that entries remain monomials in $a$.
\[
  \begin{bmatrix}
     0 & 0 & \frac{1}{3}a^2 & \frac{3}{4}a & 0 \\
     0 & 0 & 0 &  \frac{1}{4}a & 1
  \end{bmatrix}
\]
We rescale again so that trailing entries are 1.
\[
  \begin{bmatrix}
     0 & 0 & \frac{4}{9}a & 1 & 0 \\
     0 & 0 & 0 &  \frac{1}{4}a & 1
  \end{bmatrix}
\]
Now all powers of $a$ are positive and trailing entries are 1s lying in distinct columns, ensuring that the basis remains a basis after substituting $a = 0$. The new basis for $\phi_a(B|_{\GG_m})$
is $1, \frac{4}{9}at^3 + t^4, \frac{1}{4}at^4 + t^5$. Plugging in $a = 0$ shows that the limit algebra has monomial basis $1, t^4, t^5$. In terms of the statement of the lemma, $\eta(3) = 1$ and $\eta(4) = 1$ (suppressing the branch index since $m = 1$).
\end{example}

\begin{proof}[Proof of Lemma \ref{lem:one_aut_limit}]
Observe that $\phi_a(B|_{\GG_m})$ has an $\OO_{\GG_m}$ basis
\[
  (t_i + a^{-1}t_i^2)^j =
  \begin{cases}
    \binom{j}{0}t_i^j + \binom{j}{1}a^{-1}t_i^{j+1} + \cdots + \binom{j}{j}a^{-j}t_i^{2j} & \text{ if } 2j < c_i \\
    \binom{j}{0}t_i^j + \binom{j}{1}a^{-1}t_i^{j+1} + \cdots + \binom{j}{c_i - 1 - j}a^{1 + j - c_i}t_i^{c_i - 1} & \text{ if }2j \geq c_i.
  \end{cases}.
\]
We will find a new basis whose limit, obtained by plugging in $a = 0$, remains a basis. This will be a basis of $\lim\limits_{a \to 0} \phi_a(B)$. The method is a controlled Gaussian elimination: after rescaling each row so that its trailing term is $1$, we eliminate repeated trailing columns from bottom to top, preserving the lowest nonzero term in each row. The row reduction is done independently for each $t_i$, so we write out the process only for the unibranch case to unburden the notation. 

Reindexing, we now assume that $B$ has a basis of the form $1, t^{j_1}, \ldots, t^{j_{c - g - 1}}$ where $0 < j_1 < \cdots < j_{c - g - 1} < c$.
Applying the isomorphism $\phi_a : t \mapsto t + a^{-1}t^2$, we obtain a $\OO_{\GG_m}$-basis for $\phi_a(B|_{\GG_m})$ given by $1, (t + a^{-1}t^2)^{j_1}, \ldots, (t + a^{-1}t^2)^{j_{c - g - 1}}$.
For each $i$, let $r_i^{(1)}$ be the exponent on the highest power of $t$ whose coefficient in the expression for the $i$th basis element $(t + a^{-1}t^2)^{j_i}$ is nonzero.
Dividing each basis element by this coefficient, we find a new basis for $\phi_a(B|_{\GG_m})$ given by $b_0 = 1$, and
\[
  b_i^{(1)} = \sum_{\ell = j_i}^{c - 1} d_{i,\ell}^{(1)} a^{\left(r_i^{(1)} - \ell\right)} t^\ell
\]
for $i = 1, \ldots, c - g - 1$ where $d_{i,\ell}^{(1)} \in k$. Observe that
\[
  d_{i,\ell}^{(1)} = \begin{cases}
    \text{nonzero} & \text{ for }j_i \leq \ell < r_i^{(1)} \\
    1 & \text{ for }\ell = r_i^{(1)} \\
    0 & \text{ for }\ell < j_i \text{ or }\ell > r_i^{(1)}
  \end{cases}
\]
and
\[
  r_i^{(1)} = \begin{cases}
      2j_i & \text{if } 2j_i \leq c - 1 \\
      c - 1 & \text{if } 2j_i > c - 1.
  \end{cases}
\]

We arrange the coefficients of each $b_i^{(1)}$ into the rows of a matrix
\[
  \begin{bmatrix}
    d^{(1)}_{1,1}a^{r_1^{(1)} - 1} & d^{(1)}_{1,2}a^{r_1^{(1)} - 2} & \cdots & d^{(1)}_{1,c-1}a^{r_{1}^{(1)} - c +1} \\
    \vdots & \vdots & & \vdots \\
    d^{(1)}_{c - g - 1,1}a^{r_{c - g - 1}^{(1)}-1} & d^{(1)}_{c - g - 1,2}a^{r_{c - g - 1}^{(1)}-2} & \cdots & d^{(1)}_{c - g - 1,c-1}a^{r_{c - g - 1}^{(1)} - c + 1}.
  \end{bmatrix}
\]

If no value of $r_i^{(1)}$ is repeated, we conclude row reduction. Otherwise, let $r^{(1)}$ be the largest repeated value of $r_i^{(1)}$ and let $q^{(1)}$ be the largest index such that $r_{q^{(1)}}^{(1)} = r^{(1)}$. (In other words, the $r^{(1)}$th column is the rightmost column containing multiple trailing entries, and the $q^{(1)}$th row is the lowest row containing such a trailing entry.)
We perform the change of basis
\[
  b_{i}^{(2')} = \begin{cases}
      b_{i}^{(1)} - b_{q^{(1)}}^{(1)} & \text{if } i < q^{(1)} \text{ and } r_i^{(1)} = r_{q^{(1)}}^{(1)} \\
      b_i^{(1)} & \text{otherwise.}
  \end{cases}
\]
That is, we perform row replacement operations to cancel out the trailing entries in the matrix above the $(q^{(1)},r^{(1)})$th
entry. Only coefficients with the same $a$-degree are combined, so we may write $b_i^{(2')} = \sum_{\ell = 1}^{c - 1} d_{i,\ell}^{(2')} a^{r_i^{(1)} - \ell} t^\ell$. Since the $q^{(1)}$th row only had support in columns $j_{q^{(1)}}, \ldots, r_{q^{(1)}}^{(1)}$ and we only subtracted $b_{q^{(1)}}$ from rows above it, $d_{i,\ell}^{(2')} = d_{i,\ell}^{(1)}$ remains nonzero for $\ell = j_i$. Now, for each $i$, set $r_i^{(2)}$ equal to the largest $\ell$ for which $d_{i,\ell}^{(2')} \neq 0$. Then divide each basis element by the corresponding coefficient to obtain a new basis
\[
  b_i^{(2)} = \frac{b_i^{(2')}}{d_{i, r_i^{(2)}}^{(2')}a^{r_i^{(1)}-r_i^{(2)}}}
\]
for $i = 1, \ldots, c - g - 1$.

Writing $b_i^{(2)} = \sum_{\ell = 1}^{c - 1} d_{i,\ell}^{(2)} a^{\left(r_i^{(2)} - \ell\right)} t^\ell$,
observe that for each $i = 1, \ldots, c - g - 1$, we have
  \[
  d_{i,\ell}^{(2)} = \begin{cases}
    0 & \text{ for }\ell < j_i \text{ or }\ell > r_i^{(2)}\\
    1 & \text{ for }\ell = r_i^{(2)} \\
    \text{nonzero} & \text{ for }\ell = j_i.
  \end{cases}
  \]
If there are no repeated values of $r_i^{(2)}$ the process concludes. Otherwise, let $r^{(2)}$ (which is less than $r^{(1)}$) be the greatest repeated value of $r_i^{(2)}$ and repeat the process until no repeated $r_i^{(\ell)}$s remain. The process must terminate since the maximum repeated $r$-value strictly decreases at each step and is bounded below by 1.

After row reduction we have a basis for the family given by
\[
  b_i^{(\infty)} = \sum_{\ell = 1}^{c-1} d_{i,\ell}^{(\infty)} a^{\left(r_i^{(\infty)} - \ell\right)}t^\ell
\]
for $i = 1, \ldots, c - g - 1$, where
\[
  d_{i,\ell}^{(\infty)} = \begin{cases}
    0 & \text{ for }\ell < j_i \text{ or }\ell > r_i^{(\infty)}\\
    1 & \text{ for }\ell = r_i^{(\infty)} \\
    \text{nonzero} & \text{ for }\ell = j_i.
  \end{cases}
\]
and the $r_i^{(\infty)}$s are distinct. Thus, plugging in $a = 0$, the limit subalgebra has monomial basis
\[
  1, t^{r_1^{(\infty)}}, \ldots, t^{r_{c - g - 1}^{(\infty)}}.
\]

Finally, we claim that $r_i^{(\infty)} \geq j_{i+1} - 1$ for $i = 1, \ldots, c - g - 2$ and $r_{c - g - 1}^{(\infty)} = c-1$.
To see that $r_{c - g - 1}^{(\infty)} = c - 1$, observe that the bottom row is unaltered by the row replacement operations, so
\[
  r_{c - g - 1}^{(\infty)} = r_{c - g - 1}^{(1)} = \begin{cases}
      2j_{c - g - 1} & \text{ if } 2j_{c - g - 1} \leq c-1 \\
      c - 1 & \text{ if } 2j_{c - g - 1} > c - 1.
  \end{cases}
\]
The first case is impossible, since $B$ is closed under multiplication and $j_{c - g - 1}$ is the highest power of $t$ contained in $B$.

Now, to see that $r_i^{(\infty)} \geq j_{i+1} - 1$ for $i = 1, \ldots, c - g - 2$, observe that the entries of the $i$th row
of the matrix in columns $1, \ldots, j_{i+1} - 1$ are at most scaled by a common unit after row reduction, since the only rows subtracted from the
$i$th row have support contained in the columns $j_{i+1}, \ldots, c - 1$. Since the support of the $i$th row initially was in
the $j_i$th to $r_i^{(1)}$ columns, we must have $r_i^{(\infty)} \geq \min\{j_{i+1}-1, r_i^{(1)}\} = \min\{ j_{i+1} - 1, 2j_i, c - 1\}$. Since $B$ is closed under multiplication,
$2j_i$ must either be greater than $c - 1$ or equal to $j_{i'}$ for some $i' > i$. In either case, $j_{i+1}-1$ is the minimum. Thus, $r_i^{(\infty)} \geq j_{i+1} - 1$.

The claim implies that $r_i^{(\infty)} = j_i$ for all $i$ if and only if $j_1, \ldots, j_{c - g - 1}$ is $g+1, \ldots, c-1$, i.e., $B$ is a partition subalgebra.
\end{proof}

\begin{remark} \label{rem:why_char_0}
This argument cannot be made to work in characteristic $p$. To see why, suppose $k$ has characteristic $p$ and let $B = k\cdot 1 + k\cdot t^p + k \cdot t^{2p} + \cdots$ be the subalgebra of $A^+_{c}$ generated by $t^p$. We claim that $B$ is fixed by all $\phi \in G_{\vec{c}}$, so the orbit closure does not contain a partition subalgebra. Indeed, if $\phi(t) = \sum_i a_i t^i$, 
then $\phi(t^p) = \phi(t)^p = \sum_i a_i^p t^{ip} \in B$. Since $B$ is finite dimensional and $t^p$ generates $B$, it follows that $\phi(B) = B$.
\end{remark}

Combining the Lemma above with our study of limits under the action of the torus, we can now easily prove the following.

\begin{theorem} \label{thm:aut_orbit_contains_partition_sing}
Let $B \in \Ter^{g}_{\vec{c}}(k)$. Then the closure of the orbit $G^\circ_{\vec{c}} \cdot B$ contains a partition subalgebra.
\end{theorem}

\begin{proof}
Since orbit closures are unions of orbits, it is enough to show that there is a sequence of algebras $B = B_1, \ldots, B_\ell$ such that
$B_{i+1}$ belongs to the closure of the orbit of $B_i$ for each $i$ and $B_\ell$ is a partition subalgebra. By Theorem \ref{thm:fixed_points_are_monoids}, any orbit closure contains a subalgebra with a monomial basis, so we may take $B_2$ to be a subalgebra with a monomial basis. Now, let $B_{i + 1} = \lim\limits_{a \to 0} \phi_a(B_i)$ for $i = 2, 3, \ldots$. By Lemma \ref{lem:one_aut_limit}, the degrees of the basis
elements of the $B_i$s increase until for some sufficiently large $\ell$, $B_{\ell}$ is a partition subalgebra.
\end{proof}

The theorem immediately implies the following corollary, since orbit closures under an irreducible group are irreducible.

\begin{corollary}
Every irreducible component of $\Ter^g_{\vec{c}}$ contains a partition subalgebra. 
\end{corollary}
\begin{proof}
Let $Z$ be an irreducible component of $\Ter^g_{\vec{c}}$. Since $\Ter^g_{\vec{c}}$ is finite type and $k$ is algebraically closed, there is a $k$-point $x$ of $Z$ not contained in any other irreducible component of $\Ter^g_{\vec{c}}$. Since the orbit closure $\ol{G^\circ_{\vec{c}} \cdot x}$ is irreducible and contains $x$, it must be contained in $Z$. Then $Z$ contains a partition subalgebra by Theorem \ref{thm:aut_orbit_contains_partition_sing}. 
\end{proof}

We conclude with the proof of our main theorem.

\begin{theorem} \label{thm:close_connectedness}
Let $k$ be a field of characteristic 0 (not necessarily algebraically closed). Then each irreducible component of $\UU_{g,n} \times \Spec k$ intersects the substack of smoothable curves $\mathcal{V}_{g,n} \times \Spec k$.
\end{theorem}

\begin{proof}
Since each irreducible component of $\UU_{g,n} \times \ol{k}$ dominates a component of $\UU_{g,n} \times k$ and points of $\mathcal{V}_{g,n} \times \ol{k}$ map to points of $\mathcal{V}_{g,n} \times k$, it suffices to prove the statement after base change to an algebraic closure of $k$. We therefore assume that $k$ is algebraically closed.

Since \(\UU_{g,n}\times \Spec k\) is an algebraic stack locally of finite type \cite[Corollary B.4]{smyth_zstable}
over the algebraically closed field \(k\), its underlying topological space is
locally Noetherian. In particular, its irreducible components are well-defined,
and each irreducible component contains a \(k\)-point not lying on any other
irreducible component.

Let $W$ be an irreducible component of $\UU_{g,n} \times \Spec k$. Suppose $x = (C, p_1,\ldots,p_n) \in \UU_{g,n}(k)$ is a reduced, connected, proper $n$-pointed curve over $k$ belonging to $W$ but to no other irreducible component of $\UU_{g,n} \times \Spec k$. Let $\tilde{C}$ be the normalization of $C$ and let $Z$ be the conductor locus of the normalization map, isomorphic to $\Spec A_{\vec{c}}$ for some tuple of conductances $(c_1, \ldots, c_m)$. As in the proof of \cite[Theorem 1.1]{bozlee_connectedness},
there is a morphism $\prod_{P \in \mathcal{P}} \Ter^{g(P)}_{\vec{c}|_P} \to \UU_{g,n} \times \Spec k$ for some partition $\mathcal{P}$ of
$\{1,\ldots, m\}$ and function $g : \mathcal{P} \to \NN$ taking some $k$-point $\tilde{x}$ to $x$. The map takes a tuple of subalgebras $(B_P)_{P \in \mathcal{P}}$ of $\prod_{P \in \mathcal{P}} A^+_{\vec{c}|_P}$ to the curve $C_{(B_P)}$ defined as the pushout
\[
\begin{tikzcd}
  Z \ar[r] \ar[d] & \tilde{C} \ar[d] \\
  \Spec (\prod_{P \in \mathcal{P}} B_P) \ar[r] & C_{(B_P)}.
\end{tikzcd}
\]

Write $G$ for the connected component of the identity of $\prod_{P \in \mathcal{P}} G_{\vec{c}|_P} \times \Spec k$. Consider the composition $\phi : G \to \prod_{P \in \mathcal{P}} \Ter^{g(P)}_{\vec{c}|_P} \to \UU_{g,n} \times \Spec k$ where the first map takes $g \in G$ to $g \cdot \tilde{x}$. By Theorem \ref{thm:aut_orbit_contains_partition_sing}, $\ol{\phi(G)}$ contains some point $y$ corresponding to a curve with only partition singularities. Since partition singularities are smoothable, and reduced curves are smoothable if and only if their singularities are smoothable \cite[Corollary 29.10]{hartshorne_deformation_theory}, $y$ belongs to $\mathcal{V}_{g,n}$. Since $G$ is irreducible, $\ol{\phi(G)}$ is irreducible; it contains $x = \phi(e)$. Since $x$ lies on no irreducible component other than $W$, the orbit closure $\ol{\phi(G)}$ is also contained in $W$. Hence $y \in W$. Since $W$ was arbitrary, we conclude that all irreducible components of $\UU_{g,n} \times \Spec k$ intersect the irreducible component of smoothable curves.
\end{proof}

\bibliographystyle{amsalpha}
\bibliography{bibliography}

\end{document}